\documentclass[reqno,12pt]{amsart}
\usepackage{geometry}
\usepackage{amsmath,amssymb,mathrsfs,color,stackengine}
\usepackage[colorlinks,
linkcolor=red,
anchorcolor=green,
citecolor=blue,
]{hyperref}
\usepackage[upint]{stix}
\usepackage{subfigure}
\usepackage{float}
\usepackage{times}
\usepackage{tikz}
\usetikzlibrary{intersections}
\usepackage{paralist}

\usepackage{comment}

\makeatletter
\def\setliststart#1{\setcounter{\@listctr}{#1}%
  \addtocounter{\@listctr}{-1}}
\makeatother

\makeatletter
\@addtoreset{figure}{section}
\makeatother

\usepackage{calc}
 \newtheorem{The}{Theorem}[section]
 \newtheorem{Cor}[The]{Corollary}
 \newtheorem{Lem}[The]{Lemma}
 \newtheorem{Pro}[The]{Proposition}
 \theoremstyle{definition}
 \newtheorem{defn}[The]{Definition}
 \newtheorem{Rem}[The]{Remark}
 
 \numberwithin{equation}{section}

\newcommand{\R}{\mathbb{R}}

\newcommand{\N}{\mathbb{N}}

\newcommand{\Q}{\mathbb{Q}}

\newcommand{\supp}{\mathrm{supp}\,}

\newcommand{\Lip}{\mathrm{Lip\,}}

\newcommand{\law}{\mathrm{law\,}}

\title{Optimal transport in the frame of abstract Lax-Oleinik operator revisited}
\address{Department of Mathematics, Nanjing University, Nanjing 210093, China}
\email{chengwei@nju.edu.cn}
\address{School of Mathematical Sciences, Shanghai Jiao Tong University, Shanghai 200240, China}
\email{hjh9413@sjtu.edu.cn}
\address{Department of Mathematics, Nanjing University, Nanjing 210093, China}
\email{tqshi.math@gmail.com}
\author{Wei Cheng, Jiahui Hong \and Tianqi Shi}
\begin{document}
\maketitle

\begin{abstract}
This is our first paper on the extension of our recent work on the Lax-Oleinik commutators and its applications to the intrinsic approach of propagation of singularities of the viscosity solutions of Hamilton-Jacobi equations. We reformulate Kantorovich-Rubinstein duality theorem in the theory of optimal transport in terms of abstract Lax-Oleinik operators, and analyze the relevant optimal transport problem in the case the cost function $c(x,y)=h(t_1,t_2,x,y)$ is the fundamental solution of Hamilton-Jacobi equation. For further applications to the problem of cut locus and propagation of singularities in optimal transport, we introduce corresponding random Lax-Oleinik operators. We also study the problem of singularities for $c$-concave functions and its dynamical implication when $c$ is the fundamental solution with $t_2-t_1\ll1$ and $t_2-t_1<\infty$, and $c$ is the Peierls' barrier respectively. 
\end{abstract}

\section{Introduction}

This is the first paper of further extension of our recent work on the Lax-Oleinik commutators and its applications to the problem of propagation of singularities of the viscosity solutions of Hamilton-Jacobi equations \cite{Cannarsa_Cheng_Hong2023}. The relation of optimal transport and weak KAM theory was first found by Bernard and Buffoni in the papers \cite{Bernard_Buffoni2006,Bernard_Buffoni2007a,Bernard_Buffoni2007b}. The readers can refer to \cite{Villani_book2009} for a systematic treatment of the optimal transport problem for the cost of Lagrangian action. For the squared distance cost function and the associated Hopf-Lax semigroup, see, for instance \cite{Ambrosio_GigliNicola_Savare_book2008,Ambrosio_Brue_Semola_book2021}. The purpose of this paper is to revisit the optimal transport problem for Lagrangian action. It is motivated by our recent work on Lax-Oleinik commutators and its applications to the problem of cut locus and singularities in optimal transport which is not essentially well understood before. The applications of the main setting and results to this topic will be in our current papers in preparation.

\subsection{Singularities of $c$-concave functions}

Let $X$ and $Y$ be two Polish spaces and $c:X\times Y\to\R$. The function $c$ is called \emph{cost function}. We consider an abstract setting of Lax-Oleinik operators. For any $\phi:X\to\R$ and $\psi:Y\to\R$ we define the \emph{abstract Lax-Oleinik operators}
\begin{align*}
	T^-\phi(y):=\inf_{x\in X}\{\phi(x)+c(x,y)\},\qquad T^+\psi(x):=\sup_{y\in Y}\{\psi(y)-c(x,y)\},\qquad x\in X, y\in Y. 
\end{align*}

Recall that a function $\psi:Y\to[-\infty,+\infty)$ is said to be \emph{$c$-concave} if it is the infimum of a family of functions $c(x,\cdot)+\alpha(x)$. Analogously, $\phi:X\to(-\infty,+\infty]$ is said to be \emph{$c$-convex} if it is the supreme of a family of functions $\beta(y)-c(\cdot,y)$. Given $\psi:Y\to[-\infty,+\infty)$ and $\psi(y)>-\infty$ we define the $c$-superdifferential as
\begin{align*}
	\partial^c\psi(y):=\{x\in X: \psi(\cdot)-c(x,\cdot)\ \text{attains maximum at}\ y\}.
\end{align*}
For any $c$-concave function $\psi:Y\to\R$, $y\in Y$ is called a \emph{singular point} of $\psi$ if $\partial^c\psi(y)$ is not a singleton. For any $c$-concave function $\psi:Y\to\R$, we denote by $\text{Sing}^c(\psi)$ the set of all singular points of $\psi$.

By using this setting we first obtain the characterization of $c$-concavity via Lax-Oleinik commutators: the following statements are equivalent (see Theorem \ref{thm:commutator}).
\begin{enumerate}[\rm (1)]
	\item $T^-\circ T^+\psi=\psi$.
	\item $\partial^c\psi(y)\not=\varnothing$ for all $y\in Y$, where $\partial^c\psi(y):=\{x\in X: \psi(\cdot)-c(x,\cdot)\ \text{attains maximum at}\ y\}$.
	\item There exists $\phi:X\to\R$ such that $\psi=T^-\phi$.
	\item $\psi$ is $c$-concave.
\end{enumerate}	

The problem of singularities of semiconcave functions plays an important role in calculus of variation and optimal control, PDE, Hamiltonian dynamical systems and geometry. For convex Lagrangian on Euclidean space or finite dimensional manifold, the maximal regularity of the associated value function is semiconcavity. More precisely, for any Tonelli Lagrangian $L:\R\times TM\to\R$ with $M$ an $d$-dimensional smooth manifold and a function $\phi:M\to\R$, let
\begin{align*}
	u(t,x):=\inf_{\gamma}\left\{\phi(\gamma(0))+\int^t_0L(s,\gamma(s),\dot{\gamma}(s))\ ds\right\},\qquad t>0, x\in M,
\end{align*}
where $\gamma:[0,t]\to M$ is taken over the set of absolutely continuous curves such that $\gamma(t)=x$. Then the function $u$ is locally semiconcave on $(0,+\infty)\times M$. We say a point $(t,x)\in(0,+\infty)\times M$ a singular point of $u$ if $u$ is not differentiable at $(t,x)$. We denote by $\text{Sing}\,(u)$ the set of all singular points of $u$. If $u$ is only a locally semiconcave function on $M$, we also denote the singular set of $u$ by $\text{Sing}\,(u)$.  

There is a large amount of literatures on the structure of $\text{Sing}\,(u)$ from topological, measure theoretic and even to dynamical aspects. For a survey the readers can refer to \cite{Cannarsa_Cheng2021a}. The rectifiability result of $\text{Sing}\,(u)$ plays an important role in the solvability issue of Monge problem in optimal transport problem, for cost function which is the squared distance on Euclidean space or Riemannian manifold. However, for general cost functions, the relation between $\text{Sing}^c(\psi)$ and the solution of relevant Monge problem is much more complicated. In \cite{Balogh_Penso2018}, the authors discuss this problem under rather general conditions. However, invoking our recent work on the Lax-Oleinik commutators with applications to the singularities in the context of weak KAM theory, we realize it is useful to understand $\text{Sing}^c(\psi)$ and their dynamical nature.

In the context of weak KAM theory, we suppose $X=Y=M$ with $M$ a compact and connected manifold without boundary. Given any $x,y\in M$, and $t_1<t_2$, we denote by $\Gamma_{x,y}^{t_1,t_2}$ the set of absolutely continuous curves $\gamma\in \text{AC}([t_1,t_2],M)$ with $\gamma(t_1)=x$ and $\gamma(t_2)=y$. Recall that 
\begin{align*}
	h(t_1,t_2,x,y):=\inf_{\gamma\in\Gamma_{x,y}^{t_1,t_2}}\int_{t_1}^{t_2}L(s,\gamma(s),\dot\gamma(s))\,ds,\qquad t_1<t_2,\quad x,y\in M,
\end{align*}
and $h(x,y)=\liminf_{t_2-t_1\to\infty}h(t_1,t_2,x,y)$ is the Peierls' barrier with a time-independent Lagrangian $L$ and the Ma\~n\'e's critical value is $0$.

To study the singularities of $c$-concave functions, we will consider three types of cost functions in different cases: $c(x,y)=h(t_1,t_2,x,y)$ with $t_2-t_1\ll1$, $c(x,y)=h(t_1,t_2,x,y)$ with finite $t_1<t_2$ and $c(x,y)=h(x,y)$. 

A key observation for the first case is due to a theorem by Marie-Claude Arnaud (\cite{Arnaud2011}) on the evolution of the 1-graph of semiconcave functions under Hamiltonian flow for short time. In this case, we conclude $\text{\rm Sing}^c(\psi)=\text{\rm Sing}\,(\psi)$ (Proposition \ref{pro:coincide}). In the case with finite $t_1<t_2$, one can only have the inclusion $\text{Sing}^c(\psi)\subset\text{Sing}\,(\psi)$ (Proposition \ref{pro:long_time}). For the case $c(x,y)=h(x,y)$ and $u^-$ is a weak KAM solution, we prove for any $y\in M$, there exists a static class $\mathcal{A}_y\subset\partial^cu^-(y)$ (Proposition \ref{pro:static_class}). Recall that the static class is the equivalence class determined by the pseudo-distance $d(x,y)=h(x,y)+h(y,x)$ on the Aubry set $\mathcal{A}$ (\cite{Mather1993,Contreras_Paternain2002}). As a consequence, if each static class is not a singleton, then $\mbox{\rm Sing}^c(u^-)=M$. 

\subsection{An alternative formulation of Kantorovich-Rubinstein duality}


Let $\phi:X\to\R$ and $\psi:Y\to\R$. Set $c_\phi(x,y):=\phi(x)+c(x,y)$ and $c^{\psi}(x,y):=\psi(y)-c(x,y)$. For any $\mu\in\mathscr{P}(X)$ and $\nu\in\mathscr{P}(Y)$, we try to find a function $\phi:X\to\R$ and a function $\psi:Y\to\R$ such that
\begin{align}
	\int_YT^-\phi\ d\nu=&\,\inf_{\pi\in\Gamma(\mu,\nu)}\int_{X\times Y}c_\phi(x,y)\ d\pi,\label{eq:K-}\tag{K$^-$}\\
	\int_XT^+\psi\ d\mu=&\,\sup_{\pi\in\Gamma(\mu,\nu)}\int_{X\times Y}c^\psi(x,y)\ d\pi.\label{eq:K+}\tag{K$^+$}
\end{align}
In terms on abstract Lax-Oleinik operators, we also proved the equivalence of the problems \eqref{eq:K-} and \eqref{eq:K+} and well known Rubinstein-Kantorovich duality (see Theorem \ref{thm:K+-}). More precisely, 
\begin{enumerate}[--]
	\item if $\phi:X\to\R$ is a solution of \eqref{eq:K-}, then $(\phi,T^-\phi)$ is a solution of Kantorovich problem. Conversely, if $(\phi,\psi)\in I_c$ is a solution of Kantorovich problem, then $\phi$ solves \eqref{eq:K-};
	\item if $\psi:Y\to\R$ is a solution of \eqref{eq:K+}, then $(T^+\psi,\psi)$ is a solution of Kantorovich problem. Conversely, if $(\phi,\psi)\in I_c$ is a solution of Kantorovich problem, then $\psi$ solves \eqref{eq:K+},
\end{enumerate}
where
\begin{align*}
	I_c:=\{(\phi,\psi): \psi(y)-\phi(x)\leqslant c(x,y)\ \text{for all}\ x\in X, y\in Y\}.
\end{align*}
Invoking the equivalence above between problem \eqref{eq:K-} (or \eqref{eq:K+}) and the Rubinstein-Kantorovich duality, it is useful to refine the analysis of the dynamical nature of the associated optimal transport problem initiated by Bernard and Buffoni (\cite{Bernard_Buffoni2006,Bernard_Buffoni2007a,Bernard_Buffoni2007b}). 

To study the singularities in the optimal transport, we have to introduce some ``Random'' Lax-Oleinik operators (Definition \ref{defn:ROL}). However, if $X=Y=M$, a connected closed manifold,  $c$ is the fundamental solution of the associated Hamilton-Jacobi equation and the terminal marginal measure $\nu$ is Dirac. The usual deterministic optimal transport problem can be understood well in terms of the associated Hamiltonian flow. Our main result (Proposition \ref{pro:exist}) in this case is: let $y_0\in M$, $\nu=\delta_{y_0}$, $\phi\in\text{Lip}_b(M)$ and $\{\Phi_H^{t_1,t_2}\}_{t_1,t_2\in\R}$ is the associated Hamiltonian flow from $t_1$ to $t_2$. 
\begin{enumerate}[\rm (1)]
	\item For any $\rho\in\mathscr{P}(\R^d)$ with $\supp(\rho)\subset D^*T^-\phi(y_0)$\footnote{For a semiconcave function $u$, $D^*u(x)$ is the set of reachable gradients.}, there exists $\mu_\rho:=(p_x\circ\Phi_{H}^{t,0})_\#(\delta_{y_0}\times\rho)$, which makes \eqref{eq:K-} holds true;
	\item For each $\mu\in\mathscr P(M)$ which satisfies \eqref{eq:K-}, there exists $\rho_\mu\in\mathscr P(\R^d)$ with $\supp(\rho_\mu)\subset D^+T^-\phi(y_0)$, such that $\mu=(p_x\circ\Phi_H^{t,0})_{\#}(\delta_{y_0}\times\rho_\mu)$.
\end{enumerate}
In the case that $\nu$ is a general probability measure, we also have the following result. Let $\nu\in\mathscr{P}(M)$ and $\phi\in\text{Lip}\,(M)$. Then, there exists some $\mu\in\mathscr{P}(M)$ solving \eqref{eq:K-} and satisfying
\begin{align*}
	\supp(\mu)\subset\mathop{\lim\inf}_{k\to\infty}\left\{x\in M: \partial_c\phi(x)\cap C_k\neq\varnothing\right\},
\end{align*}
for some $a_k$-net $C_k$  of $\supp\,(\nu)$, where $a_k\searrow 0$ as $k\to\infty$.

\subsection{Random Lax-Oleinik operators}

The introduction of random Lax-Oleinik operators is a key step for the problem of cut locus and propagation of singularities in optimal transport. We use the word ``random'' to clarify the essential difference of the stochastic setting of this problem, where sample paths of the relevant processes is not absolutely continuos.

For a metric space $(X,d)$, let $(\mathscr{P}_p(X),W_p)$ be the $p$-Wasserstein space for $p\in[1,+\infty)$, where $W_p$ is the $p$-order Wasserstein metric. Let $s,t\in\R$ with $s<t$, $\mu\in\mathscr{P}(X)$ and $\nu\in\mathscr{P}(Y)$. We follow the setting of \cite{Villani_book2009}. The dynamical cost functional associated to $c^{s,t}$ is 
\begin{align*}
	C^{s,t}(\mu,\nu):=\inf_{\pi\in\Gamma(\mu,\nu)}\int_{X\times Y}c^{s,t}(x,y)\,d\pi=\inf_{\substack{\mathrm{law}(x)=\mu\\\mathrm{law}(y)=\nu}}\mathbb E(c^{s,t}(x,y)),
\end{align*}
where $x(\omega)$ and $y(\omega)$ in the last term are $X$-valued and $Y$-valued random variables of some probability space $(\Omega,\mathscr F,\mathbb P)$ respectively. 
	
Given $\phi\in L^0(X;(-\infty,+\infty])$ which is bounded from below, associated with $\phi$ we define a functional of $\mathscr P_p(X)$, called \textit{the potential energy associated to $\phi$} (see \cite[Definition 4.30]{Ambrosio_Gigli2013} or \cite[Example 9.3.1]{Ambrosio_GigliNicola_Savare_book2008}),
\begin{align*}
	\phi(\mu):=\int_X\phi\,d\mu,\qquad\mu\in \mathscr P_p(X).
\end{align*}

Let $X=Y=M$ and $c^{s,t}(x,y):=h(s,t,x,y)$, the fundamental solution from weak KAM theory. 
For any function $\phi:\R^d\to [-\infty,+\infty]$, $t>0$, $\mu\in\mathscr P_1(\R^d)$, we define
\begin{align*}
	P_t^-\phi(\mu)&:=\inf_{\nu\in\mathscr P_1(\R^d)}\{\phi(\nu)+C^{0,t}(\nu,\mu)\},\\
		P_t^+\phi(\mu)&:=\sup_{\nu\in\mathscr P_1(\R^d)}\{\phi(\nu)-C^{0,t}(\mu,\nu)\}.
\end{align*}
We call $P_t^{\pm}\phi$ the positive and negative \emph{random Lax-Oleinik operator}. 

Our main result is: suppose $\phi$ is a uniformly continuous function on $M$, $t>0$. Then, for any $\nu\in\mathscr P_1(M)$ the operator $P_t^-\phi(\nu)$ is finite-valued, there exist $\mu\in\mathscr P_1(M)$ and $\xi\in L_{\mu,\nu}^{0,t}$ such that
\begin{align*}
	P_t^-\phi(\nu)&=\phi(\mu)+C^{0,t}(\mu,\nu)\\
		&=\int_\Omega\phi(\xi(0,\omega))+\int_0^tL(\xi(s,\omega),\dot\xi(s,\omega))\,dsd\mathbb P\\
		&=\int_0^t\int_{T\R^d}\phi(x)+L(x,v)\,d\tilde\mu_sds,
\end{align*}
where $\tilde\mu_s:=\law(\xi(s,\cdot),\dot\xi(s,\cdot))$ is determined by the associated Euler-Lagrangian flow. Moreover, $P_t^-\phi(\nu)=T_t^-\phi(\nu)$ for any $\nu\in\mathscr P_1(M)$, i.e., for any $\nu\in\mathscr P_1(M)$, there exists $\mu\in\mathscr P_1(M)$ solving \eqref{eq:K-}. 

We remark finally if the Lagrangian $L$ is a Riemannian metric, the associated ``Random'' Hopf-Lax semigroup has already been widely used in the community of optimal transport (see, for instance, \cite{Ambrosio_Brue_Semola_book2021}). However, we emphasize the abstract and random Lax-Oleinik operators for general Lagrangian is useful for our program to deal with the problem of propagation of singularities based on our intrinsic approach of this problem in the context of calculus of variation (\cite{Cannarsa_Cheng_Fathi2017,Cannarsa_Cheng3,Cannarsa_Cheng_Fathi2021}).

The paper is organized as follows. In section 2, we introduce some necessary facts from weak KAM theory, calculus of variation and optimal transport. In Section 3, we define the abstract negative and positive Lax-Oleinik operators $T^{\pm}$ and characterize the conditions $T^-\circ T^+=Id$ and $T^+\circ T^-=Id$. We also discuss the relations between the singular set of semiconcave functions and $c$-concave function in the context of weak KAM theory when $c$ is the fundamental solutions and Peierls' barrier. Section 4 is composed of three parts. We obtain an equivalent form of the Kantorovich-Rubinstein duality using the Lax-Oleinik operators $T^{\pm}$. We also discuss some basic problems of optimal transport in the frame of Lax-Oleinik operators $T^{\pm}$. In the last subsection, we introduce the random Lax-Oleinik operators and discuss some fundamental properties of these operators in frame of abstract Lax-Oleinik operators. We give the proof of a possible known lemma in the appendix.  

\medskip

\noindent\textbf{Acknowledgements.} Wei Cheng is partly supported by National Natural Science Foundation of China (Grant No. 12231010). Jiahui Hong is partly supported by Super Postdoctoral Incentive Plan of Shanghai ``Singularity problems of Hamilton-Jacobi equations''. 

\section{Preliminaries}

\subsection{Semiconcave functions}

We first recall some basic relevant notions on semiconcavity.
\begin{enumerate}[--]
	\item Let $\Omega$ be an open convex subset of $\R^d$. A function $\phi:\Omega\to\R$ is called a semiconcave function (of linear modulus) with constant $C\geqslant0$ if for any $x\in\Omega$ there exists $p\in\R^d$ such that
	\begin{equation}\label{eq:semiconcave}
	    \phi(y)\leqslant\phi(x)+\langle p,y-x\rangle+\frac C2|x-y|^2,\qquad \forall y\in\R^d.
	\end{equation}
	\item The set of covectors $p$ satisfying \eqref{eq:semiconcave} is called the proximal superdifferential of $\phi$ at $x$ and we denote it by $D^+\phi(x)$.
	\item Similarly, $\phi$ is semiconvex if $-\phi$ is semiconcave. The set of $D^-\phi(x)=-D^+(-\phi)(x)$ is called the proximal subdifferential of $\phi$ at $x$. 
	\item The set $D^+\phi(x)$ is a singleton if and only if $\phi$ is differentiable at $x$, and $D^+\phi(x)=\{D\phi(x)\}$. A point $x$ is called a singular point of a semiconcave function $\phi$ if $D^+\phi(x)$ is not a singleton. We denote by $\text{Sing}\,(\phi)$ the set of all singular points of $\phi$.
	\item We call $p\in D^*\phi(x)$, the set of reachable gradients, if there exists a sequence $x_k\to x$ as $k\to\infty$, $\phi$ is differentiable at each $x_k$ and $p=\lim_{k\to\infty}D\phi(x_k)$. We have $D^*\phi(x)\subset D^+\phi(x)$ and $D^+\phi(x)=\text{co}\,D^*\phi(x)$. 
\end{enumerate}
For more on the semiconcave functions, the readers can refer to \cite{Cannarsa_Sinestrari_book}.

\begin{Pro}\label{pro:inf}
A function $\phi:\Omega\to\R$ is a semiconcave function with constant $C\geqslant0$ if and only if there exists a family of $C^2$-functions $\{\phi_i\}$ with $D^2\phi_i\leqslant CI$ such that
\begin{align*}
	\phi=\inf_i\phi_i.
\end{align*}
\end{Pro}

Proposition \ref{pro:inf} is an important and useful characterization of semiconcavity. Let $S$ be a compact topological space and $F:S\times\R^d\to\R$ be a continuous function such that
\begin{enumerate}[\rm (a)]
	\item $F(s,\cdot)$ is of class $C^2$ for all $s\in S$ and $\|F(s,\cdot)\|_{C^2}$ is uniformly bounded by some constant $C$,
	\item $D_xF(s,x)$ is continuous on $S\times\R^d$,
	\item $\phi(x)=\inf\{F(s,x): s\in S\}$.
\end{enumerate}
We call such a function $\phi$ a marginal function of a family of $C^2$-functions.

\begin{Pro}\label{pro:marginal}
Let $\phi(x)=\inf\{F(s,x): s\in S\}$ be a marginal functions of the family of $C^2$-functions $\{F(s,\cdot)\}_{s\in S}$. Let $M(x)=\arg\min\{F(s,x): s\in S\}$.
\begin{enumerate}[\rm (1)]
	\item $\phi$ is semiconcave with constant $C$.
	\item $D^*\phi(x)\subset Y(x)=\{D_xF(s,x): s\in M(x)\}$.
	\item For each $x$
	\begin{align*}
	    D^+\phi(x)=
		\begin{cases}
			D_xF(s,x),& M(x)=\{s\}\ \text{is a singleton};\\
			\mbox{\rm co}\,\{D_xF(s,x): s\in M(x)\},&\text{otherwise.}
		\end{cases}
	\end{align*}
\end{enumerate} 	
\end{Pro}

\begin{Rem}
If $M$ is a connected and compact smooth manifold without boundary, a function $\phi:M\to\R$ is called semiconcave if there exists a family of $C^2$-functions $\{\phi_i\}$ such that $\phi=\inf_i\phi_i$, and the Hessians of $\phi_i$'s are uniformly bounded above. The readers can refer to the appendix of the paper \cite{Fathi_Figalli2010} for more  discussion of the semiconcavity of functions on (even noncompact) manifold. However, because of the local nature of our main discussion, we will work on Euclidean space instead. 
\end{Rem}

\subsection{Fundamental solutions and Lax-Oleinik representation}

A function $L=L(t,x,v):\R\times TM\to\R$ is called time-dependent \emph{Tonelli Lagrangian}, if it is a function of class $C^2$ and satisfies the following conditions:
\begin{enumerate}[(L1)]
	\item The function $v\mapsto L(t,x,v)$ is strictly convex for all $(t,x)\in\R\times M$.
	\item There exist a superlinear function $\theta:[0,+\infty)\to[0,+\infty)$ and $L^{\infty}_{\rm loc}$ function $c_0:\R\to[0,+\infty]$, such that
	\begin{align*}
		L(t,x,v)\geqslant \theta(|v|_x)-c_0,\qquad \forall (t,x,v)\in\R\times TM.
	\end{align*}
	\item There exist $C_1,C_2>0$, such that
	\begin{align*}
		|L_t(t,x,v)|\leqslant C_1+C_2L(t,x,v),\qquad \forall (t,x,v)\in\R\times TM.
	\end{align*}
\end{enumerate}
We call the Lagrangian $H$ associated with $L$ a \emph{Tonelli Hamiltonian}, defined by
\begin{align*}
	H(t,x,p)=\sup_{v\in T_xM}\{p(v)-L(t,x,v)\},\qquad (t,x,p)\in\R\times T^*M.
\end{align*}

Given any $x,y\in M$, and $t_1<t_2$, we denote by $\Gamma_{x,y}^{t_1,t_2}$ the set of absolutely continuous curves $\gamma\in \text{AC}([t_1,t_2],M)$ with $\gamma(t_1)=x$ and $\gamma(t_2)=y$. 

\begin{defn}
Given a Tonelli Hamiltonian $H$ with $L$ the associated Tonelli Lagrangian. We call the function $h=h(t_1,t_2,x,y):\R\times\R\times M\times M\to\R$ the \emph{fundamental solution} of the associated Hamilton-Jacobi equation, defined by
\begin{align*}
	h(t_1,t_2,x,y):=\inf_{\gamma\in\Gamma_{x,y}^{t_1,t_2}}\int_{t_1}^{t_2}L(s,\gamma(s),\dot\gamma(s))\,ds,\qquad t_1<t_2,\quad x,y\in M.
\end{align*}
\end{defn}

\begin{defn}
Given a Tonelli Hamiltonian $H$ with $L$ the associated Tonelli Lagrangian. For any $\phi:M\to\R$, $x\in M$, and $t_1<t_2$, we define
\begin{align*}
	T_{t_1}^{t_2}\phi(x):=\inf_{y\in M}\{\phi(y)+h(t_1,t_2,y,x)\},\\
	\breve T_{t_1}^{t_2}\phi(x):=\sup_{y\in M}\{\phi(y)-h(t_1,t_2,x,y)\}.
\end{align*}
The families $\{T_{t_1}^{t_2}\}$ and $\{\breve T_{t_1}^{t_2}\}$ of operators are called \emph{negative and positive Lax-Oleinik evolution} respectively. 
\end{defn}

\begin{Rem}
If $\phi\in\text{\rm BUC}\,(M)$, the set of (bounded) uniformly continuous functions on $M$, both $\{T_{t_1}^{t_2}\}$ and $\{\breve T_{t_1}^{t_2}\}$ satisfy Markov property. Moreover, if $L$ is independent of $t$, then both of them are semigroups on $\text{\rm BUC}\,(M)$. We denote $T_t^-:=T_0^t$ and $T_t^+:=\breve T_0^t$ for brevity. 
\end{Rem}

\section{Optimal transport and singularities of $c$-concave function}

Let $X$ and $Y$ be two Polish spaces and $c:X\times Y\to\R$. The function $c$ is called \emph{cost function}. The problem when $c=h(t_1,t_2,\cdot,\cdot)$ for some fixed $t_1<t_2$ has very natural feature. The basic relation among optimal transport, Hamilton-Jacobi equations, Mather theory and weak KAM theory in this context has been understood by many authors (\cite{Bernard_Buffoni2007a,Bernard_Buffoni2007b,Fathi_Figalli2010}). Main part of this section is motivated by our recent work on the Lax-Oleinik commutators and singularities. We begin with the analysis without touching the measure theoretic aspect of the theory at first.


\subsection{Abstract Lax-Oleinik commutators}
  
\begin{defn}\label{abstract lax-oleinik}
For any $\phi:X\to\R$ and $\psi:Y\to\R$ we define the \emph{abstract Lax-Oleinik operators}
\begin{align*}
	T^-\phi(y):=\inf_{x\in X}\{\phi(x)+c(x,y)\},\qquad T^+\psi(x):=\sup_{y\in Y}\{\psi(y)-c(x,y)\},\qquad x\in X, y\in Y. 
\end{align*}
\end{defn}

We always assume the infimum and supremum in the definitions of $T^-\phi(y)$ and $T^+\psi(x)$ can be achieved respectively.

\begin{defn}
A function $\psi:Y\to[-\infty,+\infty)$ is said to be \emph{$c$-concave} if it is the infimum of a family of functions $c(x,\cdot)+\alpha(x)$. Analogously, $\phi:X\to(-\infty,+\infty]$ is said to be \emph{$c$-convex} if it is the supreme of a family of functions $\beta(y)-c(\cdot,y)$.
\end{defn}

It is clear that $T^-\phi$ is $c$-concave and $T^+\psi$ is $c$-convex.

\begin{defn}
Given $\phi:X\to(-\infty,+\infty]$ and $\phi(x)<+\infty$ we define the $c$-subdifferential as
\begin{align*}
	\partial_c\phi(x):=\{y\in Y: \phi(\cdot)+c(\cdot,y)\ \text{attains minimum at}\ x\}.
\end{align*}
Analogously, given $\psi:Y\to[-\infty,+\infty)$ and $\psi(y)>-\infty$ we define the $c$-superdifferential as
\begin{align*}
	\partial^c\psi(y):=\{x\in X: \psi(\cdot)-c(x,\cdot)\ \text{attains maximum at}\ y\}.
\end{align*}
\end{defn}

Now, we will concentrate on the Lax-Oleinik commutators $T^-\circ T^+$ and $T^+\circ T^-$.

\begin{Lem}\label{lem:monotone1}
Let $\phi:X\to\R$ and $\psi:Y\to\R$. 
\begin{enumerate}[\rm (1)]
	\item $T^-\circ T^+\psi\geqslant\psi$ and $T^+\circ T^-\phi\leqslant\phi$. 
	\item $T^-\circ T^+\circ T^-\phi=T^-\phi$ and $T^+\circ T^-\circ T^+\psi=T^+\psi$. In particular, both $T^-\circ T^+$ and $T^-\circ T^+$ are idempotent. 
\end{enumerate}
\end{Lem}

\begin{proof}
From the definition we have
\begin{align*}
	T^-\circ T^+\psi(y)=&\,\inf_{x\in X}\{\sup_{z\in Y}(\psi(z)-c(x,z))+c(x,y)\}\geqslant\inf_{x\in X}\{\psi(y)-c(x,y)+c(x,y)\}=\psi(y),\\
	T^+\circ T^-\phi(x)=&\,\sup_{y\in Y}\{\inf_{z\in X}(\phi(z)+c(z,y))-c(x,y)\}\leqslant\sup_{y\in Y}\{\phi(x)+c(x,y)-c(x,y)\}=\phi(x),
\end{align*}
and (1) follows. By (1) we observe
\begin{align*}
	T^-\phi\leqslant T^-\circ T^+\circ T^-\phi\leqslant T^-\phi.
\end{align*}
Thus, the inequalities above are equalities and this completes the proof of the first equality in (2). The proof of the rest of (2) is similar.
\end{proof}

\begin{Lem}\label{lem:T^-T^+1}
\hfill
\begin{enumerate}[\rm (1)]
	\item For any $y\in Y$, $T^-\circ T^+\psi(y)=\psi(y)$ if and only if $\partial^c\psi(y)\not=\varnothing$. 
	\item For any $x\in X$, $T^+\circ T^-\phi(x)=\phi(x)$ if and only if $\partial_c\phi(x)\not=\varnothing$. 
\end{enumerate}
\end{Lem}

\begin{proof}
We only prove (1) since the proof of (2) is similar. If $\partial^c\psi(y)\not=\varnothing$, then there exists $x\in X$ such that $T^+\psi(x)=\psi(y)-c(x,y)$. Thus,
\begin{align*}
	\psi(y)=T^+\psi(x)+c(x,y)\geqslant T^-\circ T^+\psi(y),
\end{align*}
and we have $T^-\circ T^+\psi(y)=\psi(y)$ since the converse inequality holds automatically by Lemma \ref{lem:monotone1}. Now, suppose $T^-\circ T^+\psi(y)=\psi(y)$, then there exists $x\in X$ such that $\psi(y)=T^-\circ T^+\psi(y)=T^+\psi(x)+c(x,y)$. This implies $\partial^c\psi(y)\not=\varnothing$.
\end{proof}

\begin{The}\label{thm:commutator}
The following statements are equivalent.
\begin{enumerate}[\rm (1)]
	\item $T^-\circ T^+\psi=\psi$ (resp. $T^+\circ T^-\phi=\phi$).
	\item $\partial^c\psi(y)\not=\varnothing$ (resp. $\partial_c\phi(x)\not=\varnothing$) for all $y\in Y$ (resp. $x\in X$).
	\item There exists $\phi:X\to\R$ (resp. $\psi:Y\to\R$) such that $\psi=T^-\phi$ (resp. $\phi=T^+\psi$).
	\item $\psi$ (resp. $\phi$) is $c$-concave (resp. $c$-conconvex).
\end{enumerate}	
\end{The}

\begin{proof}
The equivalence between (1) and (2) is immediately from Lemma \ref{lem:T^-T^+1}. The equivalence between (1) and (3) can be obtained by Lemma \ref{lem:monotone1} (2).

If $T^-\circ T^+\psi=\psi$ then $\psi$ is $c$-concave by definition. Now, suppose $\psi$ is $c$-concave which is the infimum of a family of $c$-affine functions $c(x,\cdot)+\beta(x)$. That means $\psi=T^-\beta$. This leads to $T^-\circ T^+\psi=\psi$ by the equivalence of (1) and (4).
\end{proof}

\subsection{Singularities of $c$-concave functions}

\begin{defn}
For any $c$-concave function $\psi:Y\to\R$, $y\in Y$ is called a \emph{regular point} of $\psi$ if $\partial^c\psi(y)$ is a singleton and is called \emph{singular point} of $\psi$ if  $\partial^c\psi(y)$ is not a singleton. For any $c$-concave function $\psi:Y\to\R$, we denote by $\text{Sing}^c(\psi)$ the set of all singular points of $\psi$.
\end{defn}

\begin{Lem}\label{upper}
Suppose $c$ is continuous. Then for any continuous $c$-concave function $\psi:Y\to\R$, the set-valued map $y\mapsto\partial^c\psi(y)$ is upper-semicontinuous. 
\end{Lem}

\begin{proof}
Suppose $\{y_k\}\subset Y$ is a sequence such that $\lim_{k\to\infty}y_k=y$. Let $x_k\in\partial^c\psi(y_k)$, $k\in\N$, such that $x_k\to x$ as $k\to\infty$, i.e.,
\begin{align*}
	T^+\psi(x_k)=\psi(y_k)-c(x_k,y_k),\qquad k\in\N.
\end{align*}
Since $T^+\psi$ is lower-semicontinuous we have
\begin{align*}
	T^+\psi(x)\geqslant\psi(y)-c(x,y)=\liminf_{k\to\infty}T^+\psi(x_k)\geqslant T^+\psi(x).
\end{align*}
Thus $T^+\psi(x)=\psi(y)-c(x,y)$ and this implies $x\in\partial^c\psi(y)$.
\end{proof}

\begin{defn}
A point $x\in\partial^c\psi(y)$ is called a $c$-reachable gradient of $\psi$ at $y$ if there exists $y_k\to y$ as $k\to\infty$, each $\partial^c\psi(y_k)$ is a singleton, say $\partial^c\psi(y_k)=\{x_k\}$, and $\lim_{k\to\infty}x_k=x$. We denote by $\partial^*_c\psi(y)\subset\partial^c\psi(y)$ the set of $c$-reachable gradients of $\psi$ at $y$. 
\end{defn}

For the application of the singularities of $c$-concave functions to optimal transport, we need clarify the difference of the singularities between usual semiconcave function and $c$-concave function, even when $c$ is semiconcave. For the rectifiability property of the singular set $\partial^c\psi$ of a $c$-concave function under more general assumptions, the readers can refer to \cite{Balogh_Penso2018}. We will discuss this problem for three types of specific cost functions from weak KAM theory:
\begin{enumerate}[1.]
	\item $c(x,y)=h(t_1,t_2,x,y)$ with $t_2-t_1\ll1$;
	\item $c(x,y)=h(t_1,t_2,x,y)$ with finite $t_2-t_1$;
	\item $c(x,y)=h(x,y)$ with $h$ the Peierls' barrier. 
\end{enumerate}

Let $M$ be a connected and closed manifold and $\text{SCL}\,(M)$ be the set of semiconcave functions on $M$. Now let us consider the case on the time-independent Lagrangian $L$ and Hamiltonian $H$. For any $\phi\in\text{SCL}\,(M)$, set
\begin{align*}
	t_\phi:=&\,\sup\{t>0:  T^+_t\phi\in C^{1,1}(M)\},
\end{align*}
where $T_t^+\phi:=\breve T_0^t\phi$. Moreover, we denote by
\begin{align*}
	\text{graph}\,(D^+\phi):=&\,\{(x,p)\in T^*M: x\in M, p\in D^+\phi(x)\},\\
	\text{graph}\,(DT_t^+\phi):=&\,\{(x,DT_t^+\phi(x))\in T^*M: x\in M\}.
\end{align*}
We call the set $\text{\rm graph}\,(D^+\phi)$ the \emph{pseudo-graph} of $ D\phi$ or \emph{1-graph} of $\phi$. Now, we recall a result by Marie-Claude Arnaud (\cite{Arnaud2011}) at first.

\begin{Pro}\label{pro:Arnaud}
Let $\phi\in\text{\rm SCL}\,(M)$. Then $t_\phi>0$, and when $0<t<t_\phi$, $T_t^+\phi\in C^{1,1}(M)$. Furthermore, for such $t\in(0,t_\phi]$,
	\begin{equation}\label{eq:graph_evo}
		\text{\rm graph}\,(DT_t^+\phi)=\Phi_H^{-t}(\text{\rm graph}\,(D^+\phi)),
	\end{equation}
	where $\{\Phi^s_H\}_{s\in\R}$ is the Hamiltonian flow with respect to any Tonelli Hamiltonian $H:T^*M\to\R$.
\end{Pro}

%

\begin{Rem}
The original proof of Arnaud's theorem is based on some approximation in the sense of Hausdorff distance of the graph for time-independent Lagrangian (see also \cite{Shi_Cheng_Hong2023} for the case of time dependent Lagrangian). The relation \eqref{eq:graph_evo} defines a diffeomorphism and bi-Lipschitz homeomorphism between $\text{\rm graph}\,(D^+\phi)$ and $\text{\rm graph}\,(DT^+_t\phi)$ by the Hamiltonian flow. This allows us to understand the singularities of $\phi$ by means of the latter Lipschitz graph. Except for some finer results obtain in \cite{Albano_Cannarsa1999}, this idea also leads to more topological results of the singular set of semiconcave functions and $c$-concave functions from the theory of optimal transport (\cite{Shi_Cheng_Hong2023}). 
\end{Rem}

Since Proposition \ref{pro:Arnaud} only holds for short time $t_2-t_1\ll1$, we have to consider our problem for $c(x,y)=h(t_1,t_2,x,y)$ in the following three cases: $t_2-t_1\ll1$, $t_2-t_1$ is finite and $t_2-t_1\to\infty$.

\subsubsection{$c(x,y)=h(t_1,t_2,x,y)$ with $t_2-t_1\ll1$}

\begin{Pro}\label{pro:coincide}
Let $M$ be a connected closed manifold and $\psi:M\to\R$ be a $c$-concave function. Then $\psi$ be semiconcave in usual sense. For the cost function $c(x,y)=h(t_1,t_2,x,y)$, $x,y\in M$, with $0<t_2-t_1<t_\psi$ as in Proposition \ref{pro:Arnaud}, there exists a bi-Lipschitz homeomorphism between $\partial^c\psi(y)$ and $D^+\psi(y)$ for all $y\in M$. In particular, 
\begin{align*}
	\mbox{\rm Sing}^c(\psi)=\mbox{\rm Sing}\,(\psi),
\end{align*}
where $\mbox{\rm Sing}\,(\psi)$ is the set of the points of non-differentiability of $\psi$.
\end{Pro}

\begin{proof}
By Proposition \ref{pro:Arnaud} (see also \cite[Theorem 3.14]{Shi_Cheng_Hong2023}), we have
\begin{align*}
	\Phi_{H}^{t_1,t_2}(\text{\rm graph}\,(D\breve{T}_{t_1}^{t_2}\psi))=\text{\rm graph}\,(D^+\psi),
\end{align*}
where $\{\Phi^{t_1,t_2}_H\}$ is the non-autonomous Hamiltonian flow from $t_1$ to $t_2$. Given $y\in M$, and $\forall p\in D^+\psi(y)$ we denote by $(\gamma(\cdot),p(\cdot))$ the Hamiltonian trajectory determined by endpoint condition $(\gamma(t_2),p(t_2))=(y,p)$. Recalling that $p(t_2)=L_v(t_2,\gamma(t_2),\dot{\gamma}(t_2))\in D^+\psi(y)$, it follows $0\in D^+(\psi(y)-h(t_1,t_2,\gamma(t_1),y))$ by Fermat rule, or equivalently that $y$ is a maximizer of the function $\phi(\cdot)-h(t_1,t_2,\gamma(t_1),\cdot)$. Therefore, for $x=\gamma(t_1)$ we have $\breve{T}^{t_2}_{t_1}\psi(x)=\psi(y)-h(t_1,t_2,x,y)$ which implies $x\in\partial^c\psi(y)$ since $ D\breve{T}^{t_2}_{t_1}\psi(x)=p(t_1)=L_v(t_1,\gamma(t_1),\dot{\gamma}(t_1))$. The converse direction can be handled in a similar way. For any $y\in M$, set
\begin{align*}
	\Gamma(y):=\{(x,p): x\in\partial^c\psi(y), p= D\breve{T}^{t_2}_{t_1}\psi(x)\}\subset\text{\rm graph}\,(D\breve{T}_{t_1}^{t_2}\psi). 
\end{align*}
Then, this leads to the relation
\begin{align*}
	 (y,D^+\psi(y))=\Phi_{H}^{t_1,t_2}(\Gamma(y)),\qquad y\in M.
\end{align*}
Notice that $\text{\rm graph}\,(D\breve{T}_{t_1}^{t_2}\psi)$ is the graph of a Lipschitz function. This shows there is a bi-Lipschitz homeomorphism between $\partial^c\psi(y)$ and $D^+\psi(y)$ for each $y\in M$. 
\end{proof}

\subsubsection{$c(x,y)=h(t_1,t_2,x,y)$ with finite $t_2-t_1$}

\begin{Lem}\label{lem:D^+}
Let $S$ be a compact topological space and $F:S\times\R^d\to\R$. Suppose $F$ is continuous and $F(s,\cdot)$ is uniformly semiconcave with constant $C\geqslant0$ for $s\in S$. Set $u:\R^d\to\R$,
\begin{align*}
	u(x)=\inf\{F(s,x): s\in S\},\qquad x\in\R^d.
\end{align*}
For any $x\in\R^d$, let $M(x)=\{s\in S:F(s,x)=u(x)\}$. Then $u$ is semiconcave with constant $C$ and $D^+u(x)=\mbox{\rm co}\,\{D^+_xF(s,x): s\in M(x)\}$.
\end{Lem}

Lemma \ref{lem:D^+} is a refinement of Proposition \ref{pro:marginal}. We give a proof in the appendix.

\begin{Pro}\label{pro:long_time}
Under the same assumption of Proposition \ref{pro:coincide} except for that $0<t_2-t_1<t_\psi$. Then the following statements hold.
\begin{enumerate}[\rm (1)]
	\item $D^+\psi(y)$ is a singleton implies $\partial^c\psi(y)$ is a singleton. Consequently,
	\begin{align*}
		\mathrm{Sing}^c(\psi)\subset\mathrm{Sing}\,(\psi).
	\end{align*}
	\item There exists a Lipschitz map from $D^*\psi(y)$ to $\partial^*_c\psi(y)$.
\end{enumerate}
\end{Pro}

\begin{proof}
Let $y\in M$. We suppose $y$ is a point of differentiability of $\psi$. Since $T_{t_1}^{t_2}\circ \breve T_{t_1}^{t_2}\psi=\psi$, then $h(t_1,t_2,x,\cdot)$ is differentiable at $y$ and $D_yh(t_1,t_2,x,y)$ coincides for any $x\in\arg\min\{\breve T_{t_1}^{t_2}\psi(\cdot)+h(t_1,t_2,\cdot,y)\}$ by Lemma \ref{lem:D^+}. Recall that each $D_yh(t_1,t_2,x,y)=L_v(t_2,\gamma_x(t_2),\dot{\gamma}_x(t_2))$ with $\gamma_x\in\Gamma^{t_1,t_2}_{x,y}$ the unique minimizer for $h(t_1,t_2,x,y)$. Since each $\gamma_x$ satisfies the same endpoint condition 
\begin{align*}
	\gamma_x(t_2)=y,\qquad \dot{\gamma}_x(t_2)=H_p(t_2,y,D_yh(t_1,t_2,x,y)),
\end{align*}
we conclude that there exists a unique $x\in M$ such that
\begin{align*}
	\psi(y)=T_{t_1}^{t_2}\circ \breve T_{t_1}^{t_2}\psi(y)=\breve T_{t_1}^{t_2}\psi(x)+h(t_1,t_2,x,y)
\end{align*}
and $D\psi(y)=D_yh(t_1,t_2,x,y)$. The first part of (1) implies $\text{Sing}^c(\psi)\subset\text{Sing}\,(\psi)$. This completes the proof of (1).

Now, suppose $p\in D^*\psi(y)$. Then there exists a sequence $y_k\to y$ as $k\to\infty$, $\psi$ is differentiable at each $y_k$ and $p=\lim_{k\to\infty}D\psi(y_k)$. From the first part of the proof, for each $k$ there exists a unique $x_k$ such that $\partial^c\psi(y_k)=\{x_k\}$. As in the proof of the first part, let $\gamma_{x_k}\in\Gamma^{t_1,t_2}_{x_k,y_k}$ be the unique minimal curve for $h(t_1,t_2,x_k,y_k)$. Applying Ascoli-Arzela theorem, any convergent subsequence of $\{\gamma_{x_k}\}$ has a limiting curve $\gamma\in\Gamma^{t_1,t_2}_{x,y}$ under $C^2$-topology satisfying the same endpoint condition $\gamma(t_2)=y$ and $L_v(t_2,\gamma(t_2),\dot{\gamma}(t_2))=p$. This implies the sequence $\{\gamma_{x_k}\}$ converges a curve $\gamma:[t_1,t_2]\to M$ connecting $x$ to $y$. By continuity, we have
\begin{align*}
	\psi(y)=\breve T_{t_1}^{t_2}\psi(\gamma(0))+\int^{t_2}_{t_1}L(s,\gamma(s),\dot{\gamma}(s))\ ds.
\end{align*} 
This implies $\lim_{k\to\infty}x_k=x$ and $x\in\partial^*_c\psi(y)$. Let $i$ be the inclusion from $D^*\psi(y)$ to the graph $\{(y,p):p\in D^*\psi(y)\}$. We finish the proof by observing that $\pi_x\circ\Phi^{t_1,t_2}_H\circ i$ defines a Lipschitz map from $D^*\psi(y)$ to $\partial^*_c\psi(y)$.
\end{proof}

\subsubsection{$c(x,y)=h(x,y)$ with $h$ the Peierls' barrier}

Now, consider the case when $L$ is time-independent and the cost function $c(x,y)=h(x,y)$, where $h(x,y)=\liminf_{t\to\infty}h(0,t,x,y)+c[0]t$ is the Peierls' barrier (\cite{Fathi_book}). In \cite{Mather1993}, Mather introduced a pseudo-distance on the Aubry set. We suppose $c[0]=0$ and let
\begin{align*}
	d(x,y)=h(x,y)+h(y,x),\qquad \forall x,y\in\mathcal{A}.
\end{align*}
Mather proved $d$ is a pseudo-distance on $\mathcal{A}$ and each equivalence class of the relation $d(x,y)=0$ is called a static class (see more details in \cite{Contreras_Paternain2002}).

\begin{Pro}\label{pro:static_class}
Let $L$ be a time-independent Tonelli Lagrangian, $u^-$ be a weak KAM solution and and $c=h$ with $h$ the Peierls' barrier. Then, for any $y\in M$, there exists a static class $\mathcal{A}_y\subset\partial^cu^-(y)$.
\end{Pro}

\begin{proof}
Recall that $u^-=T^-_t\circ T^+_tu^-$ for all $t\geqslant0$ (\cite{Cannarsa_Cheng_Hong2023}), where $T_t^-\phi:=T_0^t\phi$. Fix $y\in M$ and let $p\in D^*u^-(y)=D^*T^-_t\circ T^+_tu^-(y)$. Then there exists $\gamma_t:[-t,0]\to M$ with $\gamma_t(0)=y$ and $L_v(\gamma_t(0),\dot{\gamma}_t(0))=p$ such that
\begin{align*}
	u^-(y)=T^-_t\circ T^+_tu^-(y)=T^+_tu^-(\gamma_t(-t))+\int^0_{-t}L(\gamma_t(s),\dot{\gamma}_t(s))\ ds=T^+_tu^-(\gamma_t(-t))+h(-t,0,\gamma_t(-t),y).
\end{align*}
Observe that since all the curves $\gamma_t$ are extremals with the same endpoint condition. Thus, there exists $\gamma:(-\infty,0]\to M$ with $\gamma(0)=y$ and $L_v(\gamma(0),\dot{\gamma}(0))=p$ such that $\gamma_t=\gamma|_{[-t,0]}$. From classical weak KAM theory we know that $T^+_tu^-$ converges uniformly to $u^+$, the conjugate function of $u^-$, as $t\to+\infty$. Then, for any limiting point $x$ of $\gamma(-t)$ as $t\to+\infty$, we have
\begin{align*}
	u^-(y)=u^+(x)+h(x,y).
\end{align*}
Such points $x\in\mathcal{A}$, the projected Aubry set. 

For any $x,x'\in\mathcal{A}$ belonging in the same static class we have
\begin{align*}
	h(x,x')+h(x',x)=0,
\end{align*}
and
\begin{align*}
	u^+(x')-u^+(x)=h(x,x').
\end{align*}
Therefore, if $x\in\mathcal{A}$ is given in the first part of the proof, which satisfies $u^+(x)=u^-(y)-h(x,y)$, then for any $x'\in\mathcal{A}$ belonging in the same static class as $x$ we have
\begin{align*}
	u^+(x')=u^+(x)+h(x,x')=u^-(y)-h(x,y)-h(x',x).
\end{align*}
By the property of Peierls' barrier 
\begin{align*}
	h(x',y)\leqslant h(x',x)+h(x,y).
\end{align*}
It follows
\begin{align*}
	u^+(x')\leqslant u^-(y)-h(x',y),\qquad\text{or}\qquad u^-(y)-u^+(x')\geqslant h(x',y).
\end{align*}
The second inequality is indeed an equality because $(u^-,u^+)$ is an admissible Kantorovich pair. This completes the proof.
\end{proof}

The following corollary is immediate by Proposition \ref{pro:static_class} and the definition of $\text{Sing}^c(u^-)$.

\begin{Cor}
If each static class is not a singleton, then $\mbox{\rm Sing}^c(u^-)=M$.
\end{Cor}

\section{Application of abstract Lax-Oleinik operators to optimal transport}

\subsection{An alternative formulation of Kantorovich-Rubinstein duality}

For any Polish space $X$ let $\mathscr{P}(X)$ be the space of Borel probability measures on $X$. Given two Polish spaces $X$ and $Y$, we call $\pi\in\mathscr{P}(X\times Y)$ a transport plan between $\mu\in\mathscr{P}(X)$ and $\nu\in\mathscr{P}(X)$ if
\begin{align*}
	(p_X)_{\#}\pi=\mu,\qquad (p_Y)_{\#}\pi=\nu
\end{align*}
where $p_X$ and $p_Y$ are the projection from $X\times Y$ to $X$ and $Y$ respectively. We denote by $\Gamma(\mu,\nu)$ the set of all transport plans between $\mu$ and $\nu$.

For any Borel cost function $c:X\times Y\to\overline{\R}$ uniformly bounded below, we consider the following problem of Kantorovich
\begin{equation}\label{eq:Kantorovich}\tag{K}
	\mathcal{C}(\mu,\nu):=\inf\Big\{\int_{X\times Y}c(x,y)\ d\pi: \pi\in\Gamma(\mu,\nu)\Big\}
\end{equation}
We denote by $\Gamma_o(\mu,\nu)$ the set of optimal transport plans.

To formulate the important theorem of Kantorovich-Rubinstein duality, we introduce 
\begin{align*}
	I_c:=&\,\{(\phi,\psi): \psi(y)-\phi(x)\leqslant c(x,y)\ \text{for all}\ x\in X, y\in Y\},\\
	K_c:=&\,\{(\phi,\psi)\in I_c: \psi=T^-\phi, \phi=T^+\psi\}.
\end{align*}
It is obvious that
\begin{align*}
	I_c=\{(\phi,\psi):\psi\leqslant T^-\phi, \phi\geqslant T^+\psi\}\supset K_c.
\end{align*}
In the literature, the elements in $K_c$ is called \emph{admissible Kantorovich pair}.

\begin{The}[Kantorovich-Rubinstein duality]\label{thm:KBD}
\begin{equation}\label{eq:KB}\tag{KR}
	\mathcal{C}(\mu,\nu)=\sup_{(\phi,\psi)\in I_c}\Big\{\int_Y\psi\ d\nu-\int_X\phi\ d\mu\Big\}=\sup_{(\phi,\psi)\in K_c}\Big\{\int_Y\psi\ d\nu-\int_X\phi\ d\mu\Big\}.
\end{equation}
\end{The}

For our purpose, we try to reformulate Theorem \ref{thm:KBD} and connect it directly to the Lax-Oleinik operators $T^{\pm}$. Let $\phi:X\to\R$ and $\psi:Y\to\R$. Set $c_\phi(x,y):=\phi(x)+c(x,y)$ and $c^{\psi}(x,y)=\psi(y)-c(x,y)$. For any $\mu\in\mathscr{P}(X)$ and $\nu\in\mathscr{P}(Y)$, we try to find a function $\phi:X\to\R$ and a function $\psi:Y\to\R$ such that
\begin{align}
	\int_YT^-\phi\ d\nu=&\,\inf_{\pi\in\Gamma(\mu,\nu)}\int_{X\times Y}c_\phi(x,y)\ d\pi,\label{eq:K-}\tag{K$^-$}\\
	\int_YT^+\psi\ d\mu=&\,\sup_{\pi\in\Gamma(\mu,\nu)}\int_{X\times Y}c^\psi(x,y)\ d\pi.\label{eq:K+}\tag{K$^+$}
\end{align}

\begin{The}\label{thm:K+-}
\hfill
\begin{enumerate}[\rm (1)]
	\item If $\phi:X\to\R$ is a solution of \eqref{eq:K-}, then $(\phi,T^-\phi)\in I_c$ is a solution of \eqref{eq:KB}. Conversely, if $(\phi,\psi)\in I_c$ is a solution of \eqref{eq:KB}, then $\phi$ solves \eqref{eq:K-}.
	\item If $\psi:Y\to\R$ is a solution of \eqref{eq:K+}, then $(T^+\psi,\psi)\in I_c$ is a solution of \eqref{eq:KB}. Conversely, if $(\phi,\psi)\in I_c$ is a solution of \eqref{eq:KB}, then $\psi$ solves \eqref{eq:K+}.
\end{enumerate}	
\end{The}

\begin{proof}
For any $\phi:X\to\R$ we observe that
\begin{equation}\label{eq:K-_ineq}
	\int_YT^-\phi\ d\nu\leqslant\inf_{\pi\in\Gamma(\mu,\nu)}\int_{X\times Y}c_\phi(x,y)\ d\pi,
\end{equation}
since for any $\pi\in\Gamma(\mu,\nu)$,
\begin{align*}
	\int_YT^-\phi\ d\nu=\int_YT^-\phi\ d\pi\leqslant\int_{X\times Y}\phi(x)+c(x,y)\ d\pi(x,y)=\int_X\phi\ d\mu+\int_{X\times Y}c(x,y)\ d\pi.
\end{align*}
Thus inequality \eqref{eq:K-_ineq} reads
\begin{align*}
	\int_YT^-\phi\ d\nu-\int_X\phi\ d\mu\leqslant\mathcal{C}(\mu,\nu).
\end{align*}
Since $\phi$ is a solution of \eqref{eq:K-}, then the inequality above is indeed an equality.  Notice $(\phi,T^-\phi)\in I_c$ since $T^-\phi(y)-\phi(x)\leqslant c(x,y)$. This implies $(\phi,T^-\phi)$ is a solution of \eqref{eq:KB} by Theorem \ref{thm:KBD}.

Conversely, suppose $(\phi,\psi)\in I_c$ is a solution of \eqref{eq:KB}. Then
\begin{align*}
	\int_Y\psi\ d\nu\geqslant\int_X\phi\ d\mu+\int_{X\times Y}c(x,y)\ d\pi=\int_{X\times Y}c_\phi(x,y)\ d\pi.
\end{align*}
Since $\psi\leqslant T^-\phi$, then $(\phi,T^-\phi)$ is also a solution of \eqref{eq:KB}. Together with \eqref{eq:K-_ineq}, $\phi$ is a solution of \eqref{eq:K-}. This completes the proof of (1). The proof (2) is similar.
\end{proof}

\subsection{Mass transport aspects of Lax-Oleinik operator}

\begin{Lem}\label{lem:proj}
Assume that $\pi\in\mathscr{P}(X\times Y)$, then
\begin{enumerate}[\rm (1)]
	\item $p_x(\supp(\pi))\subset\supp\,((p_x)_\#\pi)$;
	\item $p_x(\supp(\pi))=\supp((p_x)_\#\pi)$ if $Y$ is compact. 
\end{enumerate}
\end{Lem}

\begin{Rem}
The compactness of $Y$ in Lemma \ref{lem:proj} (2) is essential. Assume that $X=Y=\R$ and $\pi=\Sigma_{i=1}^{+\infty}2^{-i}\delta_{(\frac{1}{i},i)}\in\mathscr{P}(\R^2)$, then
\begin{align*}
\supp(\pi)=\bigg\{\big(\frac{1}{i},i\big): i\in\N\bigg\},\qquad p_x(\supp\,(\pi))=\bigg\{\frac{1}{i}: i\in\N\bigg\}.
\end{align*}
However,
\begin{align*}
(p_x)_\#\pi=\Sigma_{i=1}^{+\infty}2^{-i}\delta_{\frac{1}{i}},\qquad \supp\,((p_x)_\#\pi)=\bigg\{\frac{1}{i}: i\in\N\bigg\}\cup\{0\}.
\end{align*}
\end{Rem}

\begin{Lem}\label{lem:K-}
Suppose $\mu\in\mathscr{P}(X)$ and $\nu\in\mathscr{P}(Y)$. Then $\phi\in\emph{Lip}_b(X)$ is a solution of \eqref{eq:K-} if and only if for all $\pi\in\Gamma_o(\mu,\nu)$,
\begin{equation}\label{eq:K-equiv}
	\supp(\pi)\subset\{(x,y)\in X\times Y: x=\arg\min\{\phi(\cdot)+c(\cdot,y)\}\}=\{(x,y): y\in\partial_c\phi(x)\}.
\end{equation}
\end{Lem}

\begin{proof}
For $\phi\in\text{\rm Lip}_b(X)$ let
\begin{align*}
	F_\phi(x,y):=\phi(x)+c(x,y)-T^-\phi(y),\qquad(x,y)\in X\times Y
\end{align*}
By the definition of $T^-\phi$, $F_\phi(x,y)\geqslant 0$ for all $(x,y)\in X\times Y$. $T^-\phi$ is upper semi-continuous, and then $F_\phi(x,y)$ is a lower semi-continuous. 

If \eqref{eq:K-} holds true, then for any optimal plan $\pi\in\Gamma_o(\mu,\nu)$, 
\begin{equation}\label{eq:integrand}
	\int_{X\times Y}F_\phi(x,y)\,d\pi=0. 
\end{equation}
Let $\pi\in\Gamma_o(\mu,\nu)$ and $(x_0,y_0)\in\supp(\pi)\setminus\{(x,y)|y\in\partial_c\phi(x)\}$. Thus $F_\phi(x_0,y_0)>0$. Due to the lower semi-continuity of $F_\phi$, for arbitrary $\varepsilon\in (0,F_\phi(x_0,y_0))$, there exists $\delta>0$ such that if $(x,y)\in B((x_0,y_0),\delta)$ then
\begin{align*}
	F_\phi(x,y)>F_\phi(x_0,y_0)-\varepsilon>0.
\end{align*}
Since $\pi(B((x_0,y_0),\delta))>0$, it follows
\begin{align*}
	\int_{X\times Y}F_\phi(x,y)\,d\pi\geqslant\int_{B((x_0,y_0),\delta)}F_\phi(x,y)\,d\pi>\int_{B((x_0,y_0),\delta)}F_\phi(x_0,y_0)-\varepsilon\,d\pi>0,
\end{align*}
which contradicts to \eqref{eq:integrand}.

Conversely, suppose \eqref{eq:K-equiv} holds true for all $\pi\in\Gamma_o(\mu,\nu)$. According to the disintegration of the measure $\pi=\int_Y\pi_y\,d\nu$, we obtain
\begin{align*}
	\int_{X\times Y} \phi(x)+c(x,y)\,d\pi &=\int_{X\times Y} [\phi(x)+c(x,y)]\cdot\mathbf{1}_{\{(x,y):y\in\partial_c\phi(x)\}}(x,y)\,d\pi\\
		&=\int_Y\int_{\{(x,y):y\in\partial_c\phi(x)\}}[\phi(x)+c(x,y)]\,d\pi_yd\nu.
\end{align*}
Recall that $T^-\phi(y)\equiv\phi(x)+c(x,y)$ for all $(x,y)$ such that $y\in\partial_c\phi(x)$. Thus,
\begin{align*}
	\int_{X\times Y} \phi(x)+c(x,y)\,d\pi=\int_Y\int_{\{(x,y):y\in\partial_c\phi(x)\}}\phi(x)+c(x,y)\,d\pi_yd\nu=\int_YT^-\phi(y)\,d\nu,
\end{align*}
which means $\phi$ is a solution of \eqref{eq:K-}.
\end{proof}

From now, we assume that $X=Y=M$, a connected and closed manifold of dimension $d$. For technical reason we suppose that $TM$ and $T^*M$, the associated tangent space and cotangent space of $M$ respectively, has a trivialization. 
We consider the cost function as $c(x,y)=h(0,t,x,y)$ for some fixed $t>0$. Recall that $T^-=T_0^t$ and $T^+=\breve T_0^t$.

\begin{Pro}\label{pro:K-}
Let $y_0\in M$, $\phi\in\emph{Lip}(M)$ and $\nu=\delta_{y_0}$, then \eqref{eq:K-equiv} holds if and only if
\begin{equation}\label{eq:supp mu}
	\supp\,(\mu)\subset\{x: y_0\in\partial_c\phi(x)\}.
\end{equation}
In particular, if $\phi$ is $c$-convex, then $\{x: y_0\in\partial_c\phi(x)\}=\partial^cT^-\phi(y_0)$.
\end{Pro}

\begin{proof}
Given $\in\text{Lip}(M)$ and $\pi\in\Gamma_o(\mu,\nu)$ such that \eqref{eq:K-equiv} holds. Since $\nu=\delta_{y_0}$, by Lemma \ref{lem:proj} we have
\begin{align*}
	p_y(\supp\,(\pi))\subset\supp\,((p_y)_\#\pi)=\supp\,(\delta_{y_0})=\{y_0\}.
\end{align*}
Together with Lemma \ref{lem:K-}, $\supp\,(\pi)\subset\{(x,y_0): y_0\in\partial_c\phi(x)\}$. Then, owing to the compactness of $M$, Lemma \ref{lem:proj} (2) implies that 
\begin{align*}
	\supp(\mu)=\supp\,((p_x)_\#\pi)\subset p_x(\{(x,y_0): y_0\in\partial_c\phi(x)\})
		=\{x: y_0\in\partial_c\phi(x)\}.
\end{align*}

Conversely, if \eqref{eq:supp mu} holds, according to Lemma \ref{lem:proj},
\begin{align*}
	p_x(\supp\,(\pi))\subset\supp((p_x)_\#\pi)=\supp\,(\mu)\subset\{x: y_0\in\partial_c\phi(x)\},
\end{align*}
which implies that
\begin{align*}
	\supp(\pi)\subset p_x^{-1}(\{x: y_0\in\partial_c\phi(x)\})=\{x: y_0\in\partial_c\phi(x)\}\times M.
\end{align*}
On the other hand,
\begin{align*}
	p_y(\supp\,(\pi))\subset\supp((p_y)_\#\pi)=\supp(\delta_{y_0})=\{y_0\}.
\end{align*}
Thus, from the analysis above,
\begin{align*}
	\supp(\pi)\subset\{(x,y_0): y_0\in\partial_c\phi(x)\}\subset\{(x,y): y\in\partial_c\phi(x)\}
\end{align*}
and this completes the proof together with Lemma \ref{lem:K-}.
\end{proof}

\begin{Pro}\label{pro:exist}
Assume that $y_0\in M$, $\nu=\delta_{y_0}$, $\phi\in\emph{Lip}(M)$ and $\{\Phi_H^{t_1,t_2}\}_{t_1,t_2\in\R}$ is the associated Hamiltonian flow from $t_1$ to $t_2$. 
\begin{enumerate}[\rm (1)]
	\item Then for any $\rho\in\mathscr{P}(\R^d)$ with $\supp(\rho)\subset D^*T^-\phi(y_0)$, there exists $\mu_\rho:=(p_x\circ\Phi_{H}^{t,0})_\#(\delta_{y_0}\times\rho)$, which makes \eqref{eq:K-} holds true;
	\item For each $\mu\in\mathscr P(M)$ which satisfies \eqref{eq:K-}, there exists $\rho_\mu\in\mathscr P(\R^d)$ with $\supp(\rho_\mu)\subset D^+T^-\phi(y_0)$, such that $\mu=(p_x\circ\Phi_H^{t,0})_{\#}(\delta_{y_0}\times\rho_\mu)$.
\end{enumerate}
\end{Pro}

\begin{proof}
Note that there exists $K>0$, such that $\Phi_H^{t,0}(M\times[-\text{Lip}\,(\phi),\text{Lip}(\phi)]^d)\subset M\times[-K,K]^d$. Since $[-K,K]^d$ is compact, by Lemma \ref{lem:proj}, we only need to check
\begin{align*}
	\supp(\mu_\rho)=\tilde p_x(\supp((\Phi_H^{t,0})_\#(\delta_{y_0}\times\rho)))\subset\{x:y_0\in\partial_c\phi(x)\},
\end{align*}
where $\tilde p_x: M\times[-K,K]^d\subset T^*M\to M$, $\tilde p_x(x,p)=x$. 
	
Suppose $x_0\in\supp((\tilde p_x\circ\Phi_H^{t,0})_\#(\delta_{y_0}\times\rho))$. Then, there exists some $p_0\in[-K,K]^d$ such that $(x_0,p_0)\in\supp\,((\Phi_H^{t,0})_\#(\delta_{y_0}\times\rho))$, which means for any $r>0$ and $B((x_0,p_0),r)\subset T^*M$,
\begin{equation}\label{eq:rho}
	0<[(\Phi_H^{t,0})_\#(\delta_{y_0}\times\rho)]\,(B((x_0,p_0),r))=(\delta_{y_0}\times\rho)\,(\Phi_H^{0,t}(B(x_0,p_0),r)).
\end{equation}
We claim that $\Phi_H^{0,t}(x_0,p_0)\in\{y_0\}\times D^*T^-\phi(y_0)$. Otherwise, by the fact $\supp\,(\delta_{y_0}\times\rho)\subset\{y_0\}\times D^*T^-\phi(y_0)$, $\Phi_H^{0,t}(x_0,p_0)\notin\supp\,(\delta_{y_0}\times\rho)$. This contradicts to \eqref{eq:rho}. Thus, there exists a unique $p^*\in D^*T^-\phi(y_0)$, $\Phi_H^{0,t}(x_0,p_0)=(y_0,p^*)$. By Theorem 6.4.9 in \cite{Cannarsa_Sinestrari_book}, if we let $\tilde p_x\circ \Phi_H^{0,s}(x_0,p_0)=\xi(s)$, $s\in[0,t]$, then
\begin{align*}
	T^-\phi(y_0)&=\inf_{\gamma(t)=y_0}\left\{\phi(\gamma(0))+\int_0^tL(s,\gamma(s),\dot\gamma(s))\,ds\right\}\\
	&=\phi(x_0)+\int_0^tL(s,\xi(s),\dot\xi(s))\,ds\\
	&=\phi(x_0)+c(x_0,y_0),
\end{align*}
which means $y_0\in\partial_c\phi(x_0)$. In summary, $\mu_\rho=(p_x\circ\Phi_{H}^{t,0})_\#(\delta_{y_0}\times\rho)$ satisfies \eqref{eq:supp mu} and \eqref{eq:K-} holds by Lemma \ref{lem:K-} and Proposition \ref{pro:K-}. This completes the proof of (1).

For the second part of the proof, we consider the following set-valued map 
\begin{align*}
	\Gamma:\{x:y_0\in\partial_c\phi(x)\}&\to T^*M\\
	x&\rightsquigarrow (x,P(x)),
\end{align*}
where $P(x):=\{p:(x,p)\in\Phi_{H}^{t,0}(y_0,D^+T^-\phi(y_0))\}$. Due to the closedness of $D^+T^-\phi(y_0))$, it is easy to check that $P(x)$ is a non-empty closed set for any $x$. Moreover, for each open subset $U\subset\R^d$, 
\begin{align*}
	\{x:P(x)\cap U\neq\varnothing\}&=\{x:\exists p\in U, (x,p)\in \Phi_{H}^{t,0}(y_0,D^+T^-\phi(y_0))\}\\
	&=p_x(\Phi_{H}^{t,0}(y_0,D^+T^-\phi(y_0))\cap (U\times \{x:y_0\in\partial_c\phi(x)\}))
\end{align*}
is a measurable subset of $M$.  By a standard measurable selection theorem (see, for example, \cite{Clarke_book2013}), there exists a measurable selection $\tilde\Gamma(x):=(x,\tilde P(x))\in (x,P(x))$. Set $\tilde \nu:=(\Phi_H^{0,t}\circ\tilde\Gamma)_{\#}\mu$. 

We claim that $\supp(\tilde\nu)\subset(y_0,D^+T^-\phi(y_0))$. If not, we assume that $(y_*,p_*)\in\supp(\tilde\nu)\setminus(y_0,D^+T^-\phi(y_0))$. Then there exists $r_0>0$ such that $B((y_*,p_*),r_0)\cap (y_0,D^+T^-\phi(y_0))=\varnothing$, which means
\begin{align*}
	\Phi_H^{t,0}(B((y_*,p_*),r_0)\cap \Phi_H^{t,0}((y_0,D^+T^-\phi(y_0))=\varnothing.
\end{align*}
Since $\Phi_H^{t,0}$ is a diffeomorphism, according to the definition of $\tilde\Gamma$, $\tilde\Gamma^{-1}(\Phi_H^{t,0}(B((y_*,p_*),r_0))=\varnothing$. Moreover, $\tilde \nu(B((y_*,p_*),r_0))>0$, because $(y_*,p_*)\in\supp(\tilde\nu)$. However, 
\begin{align*}
	0<\tilde \nu(B((y_*,p_*),r_0))=(\Phi_H^{0,t}\circ\tilde\Gamma)_{\#}\mu(B((y_*,p_*),r_0))=\mu(\tilde\Gamma^{-1}(\Phi_H^{t,0}(B((y_*,p_*),r_0))))=\mu(\varnothing)=0.
\end{align*}
This leads to a contradiction and it means $\supp(\tilde\nu)\subset(y_0,D^+T^-\phi(y_0))$.

In addition, $(p_x)_{\#}\tilde\nu=\delta_{\{y_0\}}$. Invoking the disintegration theorem, we  assert that there exists $\tilde \nu_{y_0}\in\mathscr P(\R^d)$ with $\tilde\nu=\delta_{\{y_0\}}\times \tilde \nu_{y_0}$. Thus, $\supp(\tilde \nu_{y_0})\subset D^+T^-\phi(y_0)$ since $\supp(\tilde\nu)\subset(y_0,D^+T^-\phi(y_0))$. We complete the proof by choosing $\rho_\mu:=\tilde\nu_{y_0}$.
\end{proof}

\begin{Lem}\label{lem:exist1}
Let $\{y_i\}_{i=1}^N\subset M$, $\phi\in\emph{Lip}(M)$ and $\nu=\sum_{i=1}^N\lambda_i\delta_{y_i}$, where $\lambda_i\in[0,1]$ and $\sum_{i=1}^N\lambda_i=1$, $N\in\N$. Then $\mu=\sum_{i=1}^N\lambda_i\mu_i$ solves \eqref{eq:K-}, where $\{\mu_i\}_{i=1}^N\subset\mathscr{P}(M)$, and each $\mu_i$ satisfies \eqref{eq:supp mu} for $y_0=y_i$. Moreover $\supp(\mu)\subset\cup_{i=1}^N\{x:y_i\in\partial_c\phi(x)\}$. 
\end{Lem}

\begin{proof}
We first suppose $\mu=\sum_{i=1}^N\lambda_i\mu_i$. Consider  the formula \eqref{eq:K-} in the case of $\nu_i=\delta_{y_i}$, $i=1,2,\cdots,N$:
\begin{align*}
	\int_{M}T^-\phi(y)\,d\nu_i=\inf_{\pi_i\in\Gamma\,(\mu_i,\nu_i)}\int_{M\times M}\phi(x)+c(x,y)\,d\pi_i.
\end{align*}It holds true if and only if $\supp\,(\mu_i)\subset\{x:y_i\in\partial_c\phi(x)\}$. Proposition \ref{pro:exist} guarantees the existence of such $\mu_i$. For $\pi_i\in\Gamma_o(\mu_i,\delta_{y_i})$,
\begin{align*}
	\int_MT^-\phi(y)\,d\nu&=\sum_{i=1}^N\lambda_i\int_MT^-\phi(y)\,d\delta_{y_i}\\
	&=\sum_{i=1}^N\lambda_i\int_{M\times M}\phi(x)+c(x,y)\,d\pi_i\\
	&=\int_{M\times M}\phi(x)+c(x,y)\,d\pi,
\end{align*}where $\pi:=\sum_{i=1}^N\lambda_i\pi_i$. Thus $\pi\in\Gamma_o(\mu,\nu)$ and $\mu$ satisfies \eqref{eq:K-}. For the rest of the proof, we only need to recall that $\supp\,(\mu_1+\mu_2)=\supp\,(\mu_1)\cup\supp\,(\mu_2)$ for arbitrary $\mu_1, \mu_2\in \mathscr P\,(M)$.
\end{proof}

\begin{Cor}\label{cor:exist2}
	Under the assumptions in Lemma \ref{lem:exist1}, we have
	\begin{enumerate}[\rm (1)]
		\item for any $\Pi\in\mathscr P(T^*M)$, with $\supp(\Pi)\subset\cup_{i=1}^N(y_i,D^*T^-\phi(y_i))$ and $(p_x)_{\#}\Pi=\nu$, there exists $\mu_\Pi:=(p_x\circ\Phi_H^{t,0})_{\#}\Pi$, which holds \eqref{eq:K-};
		\item for any $\mu\in\mathscr P(M)$ satisfies \eqref{eq:K-}, there exists $\Pi_\mu\in\mathscr P(T^*M)$ with $(p_x)_\#\Pi_\mu=\nu$ and $\supp(\Pi_\mu)\subset\cup_{i=1}^N(y_i,D^+T^-\phi(y_i))$, such that $\mu=(p_x\circ\Phi_H^{t,0})_{\#}\Pi_\mu$.
	\end{enumerate}
\end{Cor}
\begin{proof}
By disintegration theorem, there exist measures $\{\Pi_{y_i}\}_{i=1}^N$ with $\Pi=\sum_{i=1}^N\lambda_i\Pi_{y_i}\times\delta_{y_i}$. As a direct consequence of Lemma  \ref{lem:exist1} there exist $\mu_i$ satisfies \eqref{eq:supp mu}, then $\mu_\Pi:=\sum_{i=1}^N\mu_i$ solves \eqref{eq:K-}. 

Conversely, if $\mu\in\mathscr P(M)$ satisfies \eqref{eq:K-}, due to Proposition \ref{pro:exist}, there exists $\rho_i\in\mathscr P(\R^d)$ such that, for each $i$, the support of which contains in $D^+T^-\phi(y_i)$, $\mu=(p_x\circ\Phi_H^{t,0})_\#(\delta_{y_i}\times\rho_i)$. Defining $\Pi_\mu:=\sum_{i=1}^N\lambda_i\delta_{y_i}\times\rho_i$ we obtain
\begin{align*}
	(p_x\circ\Phi_H^{t,0})_{\#}\Pi_\mu=(p_x\circ\Phi_H^{t,0})_{\#}\left(\sum_{i=1}^N\lambda_i\delta_{y_i}\times\rho_i\right)=\sum_{i=1}^N\lambda_i(p_x\circ\Phi_H^{t,0})_{\#}\left(\delta_{y_i}\times\rho_i\right)=\mu.
\end{align*}
In addition, it is easy to check that $\supp(\Pi_\mu)\subset\cup_{i=1}^N(y_i,D^+T^-\phi(y_i))$ and $(p_x)_\#\Pi_\mu=\nu$.
\end{proof}

\begin{Lem}\label{liminf}
Let $(X,d)$ be a metric space, $\{\mu_n,\mu\}_{n=1}^\infty\subset \mathscr{P}(X)$.
\begin{enumerate}[\rm (1)]
	\item \emph{(Portmanteau)} $\mu_n\mathrel{\ensurestackMath{\stackon[1pt]{\rightharpoonup}{\scriptstyle\ast}}}\mu$ if and only if $\liminf_{n\to\infty}\mu_n(G)\geqslant \mu(G)$ for any open subset $G\subset X$;
	\item $\supp\,(\mu)\subset \mathop{\lim\inf}_{n\to\infty}\supp(\mu_n)$, if $\mu_n\mathrel{\ensurestackMath{\stackon[1pt]{\rightharpoonup}{\scriptstyle\ast}}}\mu$.
\end{enumerate}
\end{Lem}

\begin{proof}
The reader can refer \cite[Theorem 2.1]{Billingsley_book1999} for the proof of (1). (2) is a direct consequence of (1). Actually, for any $x\in\supp\,(\mu)$ and any neighborhood $U_x$ of $x$ in $X$ with $\mu\,(U_x)>0$. Portmanteau's theorem implies that for all sequences $\mu_n\mathrel{\ensurestackMath{\stackon[1pt]{\rightharpoonup}{\scriptstyle\ast}}}\mu$, 
\begin{align*}
	\mathop{\lim\inf}_{n\to\infty}\mu_n(U_x)\geqslant \mu\,(U_x)>0.
\end{align*}
By the definition of limit inferior of sets, $x\in\mathop{\lim\inf}_{n\to\infty}\supp\,(\mu_n)$.
\end{proof}

\begin{The}\label{thm:exist3}
Suppose $\nu\in\mathscr{P}(M)$ and $\phi\in\emph{Lip}\,(M)$. Then there exists some $\mu\in\mathscr{P}(M)$ solving \eqref{eq:K-} and satisfying
\begin{align*}
	\supp(\mu)\subset\mathop{\lim\inf}_{k\to\infty}\left\{x\in M: \partial_c\phi(x)\cap C_k\neq\varnothing\right\},
\end{align*}
for some $a_k$-net $C_k$  of $\supp\,(\nu)$, where $a_k\searrow 0$ as $k\to\infty$.
\end{The}

\begin{proof}
Because of the closedness of $\supp\,(\nu)$ for $\nu\in\mathscr{P}(M)$ and the compactness of $M$, $\supp(\nu)$ is a compact subset of $M$. Thus $\supp(\nu)$ has a finite $1/n$-net $C_n:=\{y_i^{(n)}\}_{i=1}^{N_n}$ for each $n\in\N^*$. Now we have $\supp(\nu)\subset\cup_{i=1}^{N_n}B(y_i^{(n)},1/n)$. 
	
Choose $\{r_i\}_{i=1}^{N_n}\subset\Q^+$ satisfying $\sum_{i=1}^{N_n}r_i=1$ and $\sum_{i=1}^{N_n}\left|r_i-\nu\,(B(y_i^{(n)},1/n))\right|\leqslant1/n$. Let $\nu_n:=\sum_{n=1}^{N_n}r_i\delta_{y_i^{(n)}}$. According to the argument in \cite[Section 6]{Billingsley_book1999}, $\nu_n\mathrel{\ensurestackMath{\stackon[1pt]{\rightharpoonup}{\scriptstyle\ast}}}\nu$, i.e.,
\begin{align*}
	\int_MT^-\phi(y)\,d\nu=\lim_{n\to\infty}\int_MT^-\phi(y)\,d\nu_n.
\end{align*}
By Proposition \ref{lem:exist1}, we can find some $\mu_n\in\mathscr{P}(M)$ and $\pi_n\in\Gamma_o(\mu_n,\nu_n)$, such that
\begin{align*}
	\int_M T^-\phi(y)\,d\nu_n=\int_{M\times M}\phi(x)+c(x,y)\,d\pi_n.
\end{align*}
and
\begin{equation}\label{eq:supp}
	\supp(\mu_n)\subset\bigcup_{i=1}^{N_n}\left\{x:y_i^{(n)}\in\partial_c\phi(x)\right\}=\{x:\partial_c\phi(x)\cap C_n\neq\varnothing\}.
\end{equation}
Note that the sequence $\{\mu_n\}_{n=1}^\infty$ is tight on $\mathscr{P}(M)$. Invoking Prokhorov theorem, we can get a convergent subsequence of $\{\mu_{n_k}\}_{k=1}^\infty$, such that $\mu_{n_k}\mathrel{\ensurestackMath{\stackon[1pt]{\rightharpoonup}{\scriptstyle\ast}}}\mu\in\mathscr{P}(M)$. The stability of the transport plan implies there exists some subsequence $\{\pi_{n_k}\}_{k=1}^\infty$ and $\pi_{n_k}\mathrel{\ensurestackMath{\stackon[1pt]{\rightharpoonup}{\scriptstyle\ast}}}\pi$ with $\pi\in\Gamma_o(\mu,\nu)$.	Thus,
\begin{align*}
	\int_MT^-\phi(y)\,d\nu&=\lim_{k\to\infty}\int_MT^-\phi(y)\,d\nu_{n_k}\\
	&=\lim_{k\to\infty}\int_{M\times M}\phi(x)+c(x,y)\,d\pi_{n_k}\\
	&=\int_{M\times M}\phi(x)+c(x,y)\,d\pi,
\end{align*}
i.e., $\mu$ solves \eqref{eq:K-}. By Lemma \ref{liminf} and \eqref{eq:supp} we have
\begin{align*}
	\supp(\mu)\subset\mathop{\lim\inf}_{k\to\infty}\supp\,(\mu_{n_k})\subset \mathop{\lim\inf}_{k\to\infty}\{x:\partial_c\phi(x)\cap C_{n_k}\neq\varnothing\}.
\end{align*}
This completes the proof by choosing $a_k=1/n_k$.
\end{proof}

\begin{Cor}
Under the assumptions in Theorem \ref{thm:exist3}, there exists  $\mu\in\mathscr P(M)$ solving \eqref{eq:K-} and satisfying
\begin{align*}
	\supp(\mu)\subset p_x\circ\Phi_H^{t,0}\,(\mathrm{graph}\,(D^*T^-\phi(\supp(\nu)))).
\end{align*}
\end{Cor}

\begin{proof}
Following the notations in the proof of Theorem \ref{thm:exist3}, for each $\nu_n$ we can pick some $\Pi_n\in\mathscr P(T^*M)$ whose support is contained in $\bigcup_{i=1}^{N_n}(y_i^{(n)},D^*T^-\phi(y_i^{(n)}))$. $\mu_n:=(p_x\circ\Phi_H^{t,0})_{\#}\Pi_n$ solves \eqref{eq:K-} by Corollary \ref{cor:exist2}. Since $\nu_n\mathrel{\ensurestackMath{\stackon[1pt]{\rightharpoonup}{\scriptstyle\ast}}}\nu$ and $\{\Pi_n\}$ is tight, $\{\Pi_n\}$ admits a subsequence $\{\Pi_{n_k}\}$ converging to some $\Pi\in\mathscr P(T^*M)$. In addition, $\mu_{n_k}\mathrel{\ensurestackMath{\stackon[1pt]{\rightharpoonup}{\scriptstyle\ast}}}\mu$ with $\mu$ solving \eqref{eq:K-} and $(p_x)_\#\Pi=\nu$. Consequently, by Lemma \ref{liminf},
\begin{align*}
	\supp(\Pi)\subset\liminf_{k\to\infty}\cup_{i=1}^{N_{n_k}}\left(y_i,D^*T^-\phi\left(y_i^{(n_k)}\right)\right)\subset \mathrm{graph}\,(D^*T^-\phi(\supp(\nu))).
\end{align*}
Combining with Lemma \ref{lem:proj} and the fact that $\Phi_H^{0,t}$ is a diffeomorphism, we get
\begin{align*}
	\supp(\mu)&=\supp((p_x\circ\Phi_H^{t,0})_{\#}\Pi)\\
		&=p_x\circ\Phi_H^{t,0}(\supp(\Pi))\\
		&\subset p_x\circ\Phi_H^{t,0}\,(\mathrm{graph}\,(D^*T^-\phi(\supp(\nu)))).
\end{align*}
This completes the proof.
\end{proof}

Finally, we discuss the Monge problem. We suppose $X$ and $Y$ are both Polish spaces.

\begin{Pro}\label{pro:map}
Assume that $\mu\in\mathscr{P}(X)$, $\nu\in\mathscr{P}(Y)$ and $\mu$ has no atom. $T: X\to Y$, $T_\#\mu=\nu$. Then the followings are equivalent.
\begin{enumerate}[\rm (1)]
	\item $T$ is an optimal transport map;
	\item There exists $\phi\in\emph{Lip}_b(X)$, such that $\int_XT^-\phi(T(x))\,d\mu=\int_X\phi(x)+c(x,T(x))\,d\mu$;
	\item There exists $\phi\in\emph{Lip}_b(X)$, such that  $T(x)\in\partial_c\phi(x)$ holds for $\mu$-a.e. $x\in X$. 
\end{enumerate}
\end{Pro}

\begin{proof}
$(1)\Rightarrow(2)$. Since $\mu$ has no atom, according to Pratelli theorem (See \cite[Theorem 2.2]{Ambrosio_Brue_Semola_book2021}), if $T$ is an optimal map,
\begin{align*}
	\int_Xc(x,T(x))\,d\mu=\inf_{\pi\in\Gamma\,(\mu,\nu)}\int_{X\times Y}c(x,y)\,d\pi.
\end{align*}
Thus for any $\phi\in\text{Lip}_b(X)$, we have
\begin{align*}
	\int_X\phi(x)+c(x,T(x))\,d\mu=\inf_{\pi\in\Gamma\,(\mu,\nu)}\int_{X\times Y}\phi(x)+c(x,y)\,d\pi.
\end{align*}
Choose some $\phi$ solving \eqref{eq:K-}, then
\begin{align*}
	\int_X\phi(x)+c(x,T(x))\,d\mu=\int_YT^-\phi(y)\,d\nu=\int_XT^-\phi(T(x))\,d\mu,
\end{align*}
with $T_\#\mu=\nu$.

$(2)\Rightarrow(3)$. We only need to observe that for any $x\in X$, $T^-\phi(T(x))\leqslant \phi(x)+c(x,T(x))$. By the equality in (2), $x=\arg\min\{\phi(\cdot)+c(\cdot,T(x))\}$ for $\mu$-a.e. $x\in X$.
	
$(3)\Rightarrow(1)$. Given $\phi\in\text{Lip}_b(X)$. We have
\begin{align*}
	\int_X T^-\phi(T(x))\,d\mu&=\int_Y T^-\phi(y)\,d\nu\\
		&\leqslant\inf_{\pi\in\Gamma(\mu,\nu)}\int_{X\times Y}\phi(x)+c(x,y)\,d\pi\\
		&\leqslant\int_{X\times Y}\phi(x)+c(x,y)\,d(\text{id}\times T)_\#\mu\\
		&=\int_X\phi(x)+c(x,T(x))\,d\mu.
\end{align*}
If (3) holds, then 
\begin{equation}\label{eq:optimal map}
	\int_XT^-\phi(T(x))\,d\mu=\int_X\phi(x)+c(x,T(x))\,d\mu=\inf_{\pi\in\Gamma(\mu,\nu)}\int_{X\times Y}\phi(x)+c(x,y)\,d\pi.
\end{equation}
It is well known that the infimum of Monge's problem is not less than the minimum of the associated Kantorovich's problem, thus $T$ in \eqref{eq:optimal map} is the optimal transport map.
\end{proof}

\begin{Rem}
The function $\phi$ in (2) or (3) of Proposition \ref{pro:map} satisfies \eqref{eq:K-}.
\end{Rem}

\subsection{Random Lax-Oleinik operators}

Suppose $(X,d)$ is a metric space. Let $(\mathscr P_p(X),W_p)$ be the $p$-Wasserstein space for $p\in[1,+\infty)$, where $W_p$ is the $p$-order Wasserstein metric, defined as
\begin{align*}
	W_p(\mu,\nu):=\left\{\inf_{\gamma\in\Gamma(\mu,\nu)}\int_{X^2}d^p(x,y)\,d\gamma\right\}^{\frac{1}{p}}.
\end{align*}
If $X$ is a Polish space, then so is $(\mathscr P_p(X),W_p)$.
	
Set $\kappa_1,\kappa_2\geqslant 0$. A function $\phi: X\to [-\infty,+\infty]$ is said to be \textit{$(\kappa_1,\kappa_2)$-Lipschitz in the large}, if for any $x,y\in X$, we have
\begin{align*}
	|\phi(x)-\phi(y)|\leqslant\kappa_1d(x,y)+\kappa_2.
\end{align*}

\begin{Rem}\label{rem:lip in the large}
For any function $\phi: X\to\R$, we have:
\begin{enumerate}[\rm (1)]
	\item $\phi\in\Lip(X)$ implies $\phi$ is $(\Lip(\phi),0)$-Lipschitz in the large;
	\item if $X$ is compact, $\phi$ is $(\kappa_1,\kappa_2)$-Lipschitz in the large if and only if $\phi$ is bounded;
	\item if $X$ is a geodesic space, $\phi\in\mathrm{UC}\,(X)$ if and only for any $\varepsilon>0$, there exists $K_\varepsilon>0$, such that $\phi$ is $(\varepsilon,K_\varepsilon)$-Lipschitz in the large.
\end{enumerate}
\end{Rem}

\begin{Lem}
Assume that $\phi$ is $(\kappa_1,\kappa_2)$-Lipchitz in the large, $\mu\in\mathscr P_1(\R^d)$. Then $\phi\in L^1(\mu)$. Particularly, if $\phi\in \mathrm{UC}\,(\R^d)$ and $\mu\in\mathscr P_1(\R^d)$, then $\phi\in L^1(\mu)$.
\end{Lem}

\begin{proof}
The proof is directly from the definition of Lipschitz in the large. Indeed, for arbitrary $x_0\in X$,
\begin{align*}
	\int_{\R^d}|\phi(x)-\phi(x_0)|\,d\mu\leqslant\int_{\R^d}(\kappa_1d(x,x_0)+\kappa_2)\,d\mu<+\infty.
\end{align*}
Thus, this completes the proof with Remark \ref{rem:lip in the large}.
\end{proof}

\begin{defn}[\protect{\cite{Ambrosio_Gigli2013,Ambrosio_GigliNicola_Savare_book2008}}, Potential energy]
Given $\phi\in L^0(X;(-\infty,+\infty])$ which is bounded from below. Associated with $\phi$ we define a functional of $\mathscr P_p(X)$, called \textit{the potential energy associated to $\phi$},
\begin{align*}
	\Phi(\mu):=\int_X\phi\,d\mu,\qquad\mu\in \mathscr P_p(X).
\end{align*}
We denote this functional by $\phi(\mu)$ for short without getting confused.
\end{defn}

\begin{defn}[Dynamical cost functional]Let $s,t\in\R$ with $s<t$, $\mu\in\mathscr{P}(X)$ and $\nu\in\mathscr{P}(Y)$. The dynamical cost functional associated to $c^{s,t}$ is 
\begin{align*}
	C^{s,t}(\mu,\nu):=\inf_{\pi\in\Gamma(\mu,\nu)}\int_{X\times Y}c^{s,t}(x,y)\,d\pi=\inf_{\substack{\mathrm{law}(x)=\mu\\\mathrm{law}(y)=\nu}}\mathbb E(c^{s,t}(x,y)),
\end{align*}
where $x(\omega)$ and $y(\omega)$ in the last term are $X$-valued and $Y$-valued random variables of some probability space $(\Omega,\mathscr F,\mathbb P)$ respectively. 
\end{defn}
We remark that Brenier \cite{Brenier1987} and Knott-Smith \cite{Knott_Smith1984} proved the existence of a probability space $(\Omega,\mathscr F,\mathbb P)$ such that for arbitrary $\mu\in\mathscr P_2(\R^d)$, we can always find a random variable $T\in L^0(\Omega;\R^d)$, which satisfies $\law(T)=\mu$. Indeed, we choose $\Omega=\R^d$, $\mathscr F$ is the set of Lebesgue measurable subsets of $\R^d$ and $\mathbb P=\mathscr L^d|_{B(0,R)}$, where $R>0$ makes $\mathscr L^d(B(0,R))=1$. If $c(x,y)=\frac{1}{2}|x-y|^2$, for the corresponding Monge's transport problem, there exists an optimal map $T\in L^0(\R^d,\R^d)$ of $\mathcal C(\mathbb P,\mu)$ such that $T_\#\mathbb P=\mu$, i.e., $\law(T)=\mu$. By Theorem 10.28 and Example 10.36 of \cite{Villani_book2009}, the argument for the case $X=Y=M$ with $M$ compact manifold and $\mu\in\mathscr P(M)$ is similar.

Let us recall some classical results on measurable selection in the field of measure theory, infinite dimensional analysis and optimal transport. 


\begin{Lem}[\protect{\cite[Theorem 18.19]{Aliprantis_Border_book2006}}]\label{sele_infinite}
Let $X$ be a separable metrizable space and $(S,\Sigma)$ a measurable space. Let $\varphi: S\leadsto X$ be a weakly measurable (see \cite[Section 17]{Aliprantis_Border_book2006}) nonempty compact set-valued map, and suppose $F: S\times X\to\R$ is a Carath\'eodory function. Define the marginal functions $f: S\to \R$ by
\begin{align*}
	f(s):=\max_{x\in\varphi(s)}F(s,x),
\end{align*}
and the correspondence $\Lambda: S\leadsto X$ of maximizers by $\Lambda(s):=\mathop{\arg\max}_{x\in\varphi(s)}\{F(s,x)\}$, then
\begin{enumerate}[\rm (1)]
	\item $f$ is measurable;
	\item $\Lambda(s)$ is a nonempty compact set for each $s\in S$ and exists a measurable selection $\lambda(s)$.
\end{enumerate}
\end{Lem}

\begin{Lem}[\protect{\cite[Corollary 5.22]{Villani_book2009}}]\label{sele_villani}
Let $X$ and $Y$ be Polish spaces and let $c: X\times Y\to\R$ be a continuous cost function with a lower bound. Suppose the map
\begin{align*}
	\Omega&\to\mathscr{P}(X)\times\mathscr{P}(Y)\\
	\omega&\mapsto (\mu_\omega,\nu_\omega)
\end{align*}
is measurable. Then there is a measurable selection $\omega\mapsto \pi_\omega$ such that for each $\omega$, $\pi_\omega$ is an optimal plan between $\mu_\omega$ and $\nu_\omega$.
\end{Lem}

In the sequel, we consider the case of $X=Y=M$, a connected and closed manifold of dimension $d$. and $c^{s,t}(x,y):=h(s,t,x,y)$, the fundamental solution from weak KAM theory. 

\begin{defn}[Random Lax-Oleinik operator]\label{defn:ROL}
	For any function $\phi:M\to [-\infty,+\infty]$, $t>0$, $\mu\in\mathscr P(M)$, we define
	\begin{align*}
		P_t^-\phi(\mu)&:=\inf_{\nu\in\mathscr P(M)}\{\phi(\nu)+C^{0,t}(\nu,\mu)\},\\
		P_t^+\phi(\mu)&:=\sup_{\nu\in\mathscr P(M)}\{\phi(\nu)-C^{0,t}(\mu,\nu)\}.
	\end{align*}
\end{defn}

\begin{Pro}[\protect{\cite[Theorem 7.21]{Villani_book2009}}]
Let $s,t\in\R$ and $\mu,\nu\in\mathscr P(M)$ such that $C^{s,t}(\mu,\nu)$ is finite. Then
\begin{align*}
		C^{s,t}(\mu,\nu)&=\min_{\pi\in\Gamma(\mu,\nu)}\int_{M^{2}}h(s,t,x,y)\,d\pi=\min_{\xi\in L_{\mu,\nu}^{s,t}}\int_\Omega\int_s^tL(\xi(r,\omega),\dot\xi(r,\omega))\,drd\mathbb P,
\end{align*}where $L_{\mu,\nu}^{s,t}:=\{\xi\in L^0(\Omega;AC([s,t];M)):\mathrm{law}\,(\xi(s,\cdot))=\mu,\mathrm{law}\,(\xi(t,\cdot))=\nu\}$.
\end{Pro}


\begin{The}\label{thm:P-}
Suppose $\phi\in C(M)$ and $t>0$. Then, for any $\nu\in\mathscr P(M)$ the operator $P_t^-\phi(\nu)$ is finite-valued. There exist $\mu\in\mathscr P(M)$ and $\xi\in L_{\mu,\nu}^{0,t}$ such that
\begin{align*}
	P_t^-\phi(\nu)&=\phi(\mu)+C^{0,t}(\mu,\nu)\\
		&=\int_\Omega\phi(\xi(0,\omega))+\int_0^tL(\xi(s,\omega),\dot\xi(s,\omega))\,dsd\mathbb P\\
		&=\int_0^t\int_{TM}\phi(x)+L(x,v)\,d\tilde\mu_sds,
\end{align*}
where $\tilde\mu_s:=\law(\xi(s,\cdot),\dot\xi(s,\cdot))\in\mathscr P(TM)$ is determined by the associated Euler-Lagrangian flow. 
	
Moreover, $P_t^-\phi(\nu)=T_t^-\phi(\nu)$ for any $\nu\in\mathscr P(M)$, i.e., for any $\nu\in\mathscr P(M)$, there exists $\mu\in\mathscr P(M)$ solving \eqref{eq:K-} . 
\end{The}

\begin{proof}
It is worth noting that for arbitrary $x,y\in M$, $T_t^-\phi(y)\leqslant \phi(x)+h(0,t,x,y)$. Thus, given $\nu\in\mathscr P(M)$, for all $\mu\in\mathscr P(M)$ and $\pi\in\Gamma(\mu,\nu)$,
\begin{align*}
	\int_{M}T_t^-\phi(y)\,d\nu=\int_{M\times M}T_t^-\phi(y)\,d\pi\leqslant\int_{M\times M}\phi(x)+h(0,t,x,y)\,d\pi,
\end{align*}
which means
\begin{equation}\label{leq_p}
	T_t^-\phi(\nu)=\int_{M}T_t^-\phi(y)\,d\nu\leqslant\inf_{\mu\in\mathscr P(M)}\inf_{\pi\in\Gamma(\mu,\nu)}\int_{M\times M}\phi(x)+h(0,t,x,y)\,d\pi=P_t^-\phi(\nu).
\end{equation}
On the other hand, we can find a $M$-valued random variable $y(\omega)\in L^0(\Omega;M)$ of  $(\Omega,\mathscr F,\mathbb P)$ with $\mathrm{law\,}(y)=\nu$. Recall that, due to a priori estimates on the minimizer (for example, see \cite[Lemma 3.1]{CCJWY2020}), the marginal function can be rewritten as
\begin{align*}
	T_t^-\phi(y(\omega))=\min_{x\in B(y(\omega),\lambda t)}\{\phi(x)+h(0,t,x,y(\omega))\},\qquad\forall\omega\in\Omega,
\end{align*}
where $\lambda>0$ only depends on the Lagrangian and Lipschitz in the large constant of $\phi$. It is easy to check that $\omega\mapsto B(y(\omega),\lambda t)$ is a measurable and  set-valued map with compact value. According to Lemma \ref{sele_infinite}, it admits a measurable selection
\begin{align*}
	\omega\mapsto x(\omega)\in \mathop{\arg\min}_{x\in B(y(\omega),\lambda t)}\{\phi(x)+h(0,t,x,y(\omega))\}\subset B(y(\omega),\lambda t).
\end{align*}
Therefore, define $\mu:=x_\#\mathbb P$, then $\mu\in\mathscr P(M)$ and
\begin{align*}
	\int_{M}T_t^-\phi(y)\,d\nu &=\int_\Omega T_t^-\phi(y(\omega))\,d\mathbb P\\
		&=\int_\Omega\phi(x(\omega))+h(0,t,x(\omega),y(\omega))\,d\mathbb P\\
		&=\int_{M\times M}\phi(x)+h(0,t,x,y)\,d(x,y)_\#\mathbb P\geqslant\phi(\mu)+C^{0,t}(\mu,\nu).
\end{align*}

Since $\omega\mapsto (x(\omega),y(\omega))$ is measurable, invoking Lemma \ref{sele_villani} and the existence of random curves, we can find some $\xi(\omega)\in\Gamma_{x(\omega),y(\omega)}^{0,t}$, which satisfies
\begin{align*}
	T_t^-\phi(y(\omega))=\phi(x(\omega))+\int_0^tL(\xi(s,\omega),\dot\xi(s,\omega))\,ds
\end{align*}
and $\xi\in L^0(\Omega;\Lip([0,t];M))$.
As a consequence we obtain
\begin{align*}
	\int_{\R^d}T_t^-\phi\,d\nu\geqslant\phi(\mu)+C^{0,t}(\mu,\nu)\geqslant P_t^-\phi(\nu).
\end{align*}
Together with \eqref{leq_p}, 
\begin{align*}
	\int_{\R^d}T_t^-\phi\,d\nu=\phi(\mu)+C^{0,t}(\mu,\nu)=P_t^-\phi(\nu)=\int_\Omega\phi(\xi(0,\omega))+\int_0^tL(\xi(s,\omega),\dot\xi(s,\omega))\,dsd\mathbb P.
\end{align*}

Finally we remark that if $L(x,\cdot)$ is strictly convex, then $\xi(\cdot,\omega)\in C^2([0,t];M)$ for any $\omega\in\Omega$, and $(\xi(\cdot,\omega),\dot\xi(\cdot,\omega))\in TM$ solves the Euler-Lagrange equation. 
\end{proof}

An similar reasoning of the proof of Theorem \ref{thm:P-} leads to 

\begin{The}
Suppose $\phi\in C(M)$, $t>0$. Then, for any $\mu\in\mathscr P(M)$, $P_t^+\phi(\mu)$ is finite-valued and there exists $\nu\in\mathscr P(M)$ and $\xi\in L_{\mu,\nu}^{0,t}$ such that
\begin{align*}
	P_t^+\phi(\nu)&=\phi(\nu)-C^{0,t}(\mu,\nu)\\
		&=\int_\Omega\phi(\xi(t,\omega))-\int_0^tL(\xi(s,\omega),\dot\xi(s,\omega))\,dsd\mathbb P\\
		&=\int_0^t\int_{TM}\phi(x)-L(x,v)\,d\tilde\nu_sds,
\end{align*}
where $\tilde\nu_s:=\law(\xi(s,\cdot),\dot\xi(s,\cdot))\in\mathscr P(TM)$ is determined by the associated Euler-Lagrangian flow. 

Moreover, $P_t^+\phi(\mu)=T_t^+\phi(\mu)$ for any $\mu\in\mathscr P(M)$, i.e., for any $\mu\in\mathscr P(M)$ there exists $\nu\in\mathscr P(M)$ solving \eqref{eq:K+}. 
\end{The}

\appendix

\section{Proofs of some Lemmata}

\begin{proof}[Proof of Lemma \ref{lem:D^+}]
By some basic results from the theory of semiconcave functions, for any $s\in S$, $x\in\R^d$ and $p\in D^+_xF(s,x)$, define $\phi(s,x,\cdot)=F(t,x)+\langle p,\cdot-x\rangle+\frac C2|\cdot-x|^2$. Then, $\phi(s,x,\cdot)$ touches $F(s,\cdot)$ from above at $x$ and $p=D\phi(s,x,x)$. Thus
\begin{align*}
	F(s,x)=\inf_{x\in\R^d}\phi(s,x,x),\qquad u(x)=\inf_{s\in S,x\in\R^d}\phi(s,x,x).
\end{align*}
Then, our conclusion is a direct consequence of the property of a standard marginal function (see, for instance, \cite[Theorem 3.4.4]{Cannarsa_Sinestrari_book}).
\end{proof}

\begin{proof}[Proof of Lemma \ref{lem:proj}]
Let $x_0\in p_x(\supp\,(\pi))$, then there exists $y_0\in Y$, $(x_0,y_0)\in \supp\,(\pi)$, i.e. for any neighborhood $U_{x_0,y_0}$ of $(x_0,y_0)$ in $X\times Y$, $\pi\,(U_{x_0,y_0})>0$. Let $U_{x_0}$ be any neighborhood of $x_0\in X$. Since $U_{x_0}\times Y$ is also a neighborhood of $(x_0,y_0)$, we have
\begin{equation}\label{ux_0}
	(p_x)_\#\pi\,(U_{x_0})=\pi\,(p_x^{-1}(U_{x_0}))=\pi\,(U_{x_0}\times Y)>0,
\end{equation}
which implies $x_0\in\supp((p_x)_\#\pi)$. 

Now, we turn to the proof of (2). We only need to prove the inclusion $\supp((p_x)_\#\pi)\subset p_x(\supp(\pi))$. Let $x_0\in \supp((p_x)_\#\pi)$, then \eqref{ux_0} holds true for any neighborhood $U_{x_0}$ of $x_0\in X$. Thus, it suffices to prove that there exists $y_0\in Y$ such that $\pi\,(U_{x_0,y_0})>0$ for any $U_{x_0,y_0}\subset X\times Y$. We prove by contradiction. Assume that for any $y\in Y$ there exists an open set $V_{x_0,y}$ containing $(x_0,y)$ such that $\pi\,(V_{x_0,y})=0$. More precisely, for each $y$, there exist $R_y^X, R_y^Y>0$ with $B(x_0,R_y^X)\times B(y,R_y^Y)\subset V_{x_0,y}$, and $\pi\,(B(x_0,R_y^X)\times B(y,R_y^Y))=0$. Due to the compactness of $Y$, the open cover $\{B(x_0,R_y^X)\times B(y,R_y^Y)\}_{y\in Y}$ of $\{x_0\}\times Y$ admits a finite subcover $\{B(x_0,R_{y_i}^X)\times B(y_i,R_{y_i}^Y)\}_{i=1}^N$.
Set $R_{x_0}:=\min_{1\leqslant i\leqslant N}\{R_{y_i}^X,R_{y_i}^Y\}>0$. Then
\begin{align*}
	\{x_0\}\times Y\subset B(x_0,R_{x_0})\times Y\subset\bigcup_{i=1}^N
        \left\{B(x_0,R_{y_i}^X)\times B(y_i,R_{y_i}^Y)\right\}.
\end{align*}
It follows that
\begin{align*}
	\pi\,(B(x_0,R_{x_0})\times Y)&\leqslant \pi\left(\bigcup_{i=1}^N
        \left\{B(x_0,R_{y_i}^X)\times B(y_i,R_{y_i}^Y)\right\}\right)\\
        &\leqslant \sum_{i=1}^N\pi\left(B(x_0,R_{y_i}^X)\times B(y_i,R_{y_i}^Y)\right)=0.
\end{align*}
Therefore, $B(x_0,R_{x_0})$, as a neighborhood of $x_0$, does not satisfy \eqref{ux_0}. This leads to a contradiction. 
\end{proof}

\bibliographystyle{plain}
\bibliography{mybib}

\end{document}